\documentclass[11pt]{amsart}

\usepackage{amsmath,amsfonts,amssymb,amsthm,mathtools}

\usepackage{verbatim}
\usepackage{amsmath,amsfonts}
\usepackage[mathscr]{euscript} 
\usepackage{amsthm}
\usepackage{blkarray,bigstrut} 
\usepackage{url}

\usepackage{graphicx} 

\usepackage{enumitem} 

\usepackage{hyperref} 
\hypersetup{colorlinks} 

\usepackage[colorinlistoftodos]{todonotes}

\RequirePackage{cleveref}
\usepackage{hypcap}
\hypersetup{colorlinks=true, citecolor=darkblue, linkcolor=darkblue}
\definecolor{darkblue}{rgb}{0.0,0,0.7}
\newcommand{\darkblue}{\color{darkblue}}

\definecolor{darkred}{rgb}{0.68,0,0}
\newcommand{\darkred}{\color{darkred}}

\definecolor{darkgreen}{rgb}{0,.38,0}
\newcommand{\darkgreen}{\color{darkgreen}}

\newcommand{\defn}[1]{\emph{\darkblue #1}}
\newcommand{\defna}[1]{\emph{\darkred #1}}
\newcommand{\defnb}[1]{\emph{\darkblue #1}}
\newcommand{\defng}[1]{\emph{\darkgreen #1}}

\setlist[enumerate]{
	label=\textnormal{({\roman*})},
	ref={\roman*}}

\makeatletter
\def\th@plain{%
	\thm@notefont{}
	\itshape 
}
\def\th@definition{%
	\thm@notefont{}
	\normalfont 
}
\makeatother

\makeatletter
\def\fdsy@scale{1}
\newcommand\fdsy@mweight@normal{Book}
\newcommand\fdsy@mweight@small{Book}
\newcommand\fdsy@bweight@normal{Medium}
\newcommand\fdsy@bweight@small{Medium}

\DeclareFontFamily{U}{FdSymbolB}{}
\DeclareFontShape{U}{FdSymbolB}{m}{n}{
	<-7.1> s * [\fdsy@scale] FdSymbolB-\fdsy@mweight@small
	<7.1-> s * [\fdsy@scale] FdSymbolB-\fdsy@mweight@normal
}{}
\DeclareFontShape{U}{FdSymbolB}{b}{n}{
	<-7.1> s * [\fdsy@scale] FdSymbolB-\fdsy@bweight@small
	<7.1-> s * [\fdsy@scale] FdSymbolB-\fdsy@bweight@normal
}{}

\makeatother

\DeclareSymbolFont{fdrelations}{U}{FdSymbolB}{m}{n}
\SetSymbolFont{fdrelations}{bold}{U}{FdSymbolB}{b}{n}

\DeclareMathSymbol{\lescc}{\mathrel}{fdrelations}{66}

\newtheorem{thm}{Theorem}[section]
\newtheorem{lemma}[thm]{Lemma}

\newtheorem{cor}[thm]{Corollary}
\newtheorem{prop}[thm]{Proposition}
\newtheorem{conj}[thm]{Conjecture}
\newtheorem{conv}[thm]{Convention}

\theoremstyle{definition}
\newtheorem{ex}[thm]{Example}

\newtheorem{rem}[thm]{Remark}

\numberwithin{figure}{section}
\numberwithin{equation}{section}

\def\emp{\nothing}

\def\nn{\mathbb N}

\def\rr{\mathbb R}

\def\sm{\smallsetminus}

\def\al{\alpha}

\def\ve{\varepsilon}

\def\cF{\mathcal F}

\def\cL{U}
\def\cS{\mathcal S}

\def\<{\langle}
\def\>{\rangle}

\def\Ups{\Upsilon}

\def\0{{\mathbf 0}}

\def\nothing{\varnothing}

\def\.{\hskip.06cm}
\def\ts{\hskip.03cm}

\def\lra{\leftrightarrow}

\def\di{{\small{\ts\diamond\ts}}}

\def\nin{\noindent}

\def\SP{{\textsc{\#P}}}

\def\SP{{\textsc{\#{}P}}}

\def\aF{\textrm{F}}
\def\aFr{\textrm{\em F}}

\def\aN{\textrm{N}}
\def\aNr{\textrm{\em N}}

\def\xx{\textbf{\textit{x}}}
\def\yy{\textbf{\textit{y}}}

\def\RG{\textrm{G}}

\DeclareMathOperator{\Ec}{\mathcal{E}} 

\DeclareMathOperator{\Nc}{\mathcal{N}}

\DeclareMathOperator{\Rb}{\mathbb{R}}

\DeclareMathOperator{\vb}{\mathbf{v}}

\DeclareMathOperator{\wb}{\mathbf{w}}

\DeclareMathOperator{\xb}{\mathbf{x}}

\DeclareMathOperator{\yb}{\mathbf{y}}

\newcommand{\iA}{\textnormal{A}} 
\newcommand{\iB}{\textnormal{B}} 
\newcommand{\iC}{\textnormal{C}} 
\newcommand{\iK}{\textnormal{K}} 

\newcommand{\iQ}{\textnormal{Q}} 
\newcommand{\ibQ}{{\textnormal{\bf Q}}} 

\newcommand{\iV}{\textnormal{V}} 

\newcommand{\Vol}{\textnormal{Vol}} 

\title[On the cross-product conjecture for the number of linear extensions]{On the cross-product conjecture for
\\ the number of linear extensions}

\date{\today}

\author[Swee Hong Chan \and Igor Pak \and Greta Panova]{\ Swee Hong Chan$^{\star}$ \ \. \and \ \. Igor~Pak$^{\di}$ \ \. \and \ \. Greta~Panova$^{\natural}$}

\thanks{\thinspace ${\hspace{-.45ex}}^\star$Department of Mathematics,
Rutgers University, Piscataway, NJ, 08854.
 \.  Email: \ts \texttt{sweehong.chan@rutgers.edu}}

\thanks{\thinspace ${\hspace{-.99ex}}^\di$Department of Mathematics,
UCLA, Los Angeles, CA, 90095. \.  Email: \ts \texttt{{pak}@math.ucla.edu}}

\thanks{\thinspace ${\hspace{-0.45ex}}^\natural$Department of Mathematics, USC, Los Angeles, CA 90089. \. Email: \ts \texttt{{gpanova}@usc.edu}}

\begin{document}

\begin{abstract}
We prove a weak version of the \emph{cross--product conjecture}:  \.
$\aF(k+1,\ell) \. \aF(k,\ell+1) \ge (\frac12+\ve) \. \aF(k,\ell) \. \aF(k+1,\ell+1)$,
where \ts $\aF(k,\ell)$ \ts is the number of linear extensions for which the
values at fixed elements $x,y,z$ are $k$ and $\ell$ apart, respectively, and where
\ts $\ve>0$ \ts depends on the poset.  We also prove the converse inequality and
disprove the \emph{generalized cross--product conjecture}.  The proofs use
geometric inequalities for mixed volumes and combinatorics of words.
\end{abstract}

	\maketitle


\section{Introduction} \label{s:intro}

%
%

This paper is centered around the \defna{cross--product conjecture} (CPP) by
Brightwell, Felsner and Trotter that gives the best known bound for the
celebrated \. \defng{$\frac{1}{3}$--$\frac{2}{3}$ \. Conjecture} \ts \cite[Thm~1.3]{BFT}.
Here we prove several weak versions of the conjecture,
and disprove a stronger version we conjectured earlier in \cite{CPP1}.

\smallskip

Let \ts $P=(X,\prec)$ \ts be a poset with \ts $|X|=n$ \ts elements.
A \defn{linear extension} of $P$ is a bijection \. $L: X \to [n]=\{1,\ldots,n\}$,
such that
\. $L(x) < L(y)$ \. for all \. $x \prec y$.
Denote by \ts $\Ec(P)$ \ts the set of linear extensions of $P$.
%
Fix distinct elements \. $x,y,z\in X$.
For $k,\ell \geq 1$, let
\[
\cF(k,\ell) \ := \ \big\{L \in \Ec(P) \, : \, L(y)-L(x)=k, \.  L(z)-L(y)=\ell \big\},
\]
and let \. $\aF(k,\ell) \. := \. \big|\cF(k,\ell)\big|$.

\smallskip

\begin{conj}[{\rm \defng{Cross--product conjecture}~\cite[Conj.~3.1]{BFT}}{}]\label{conj:CPC}
We have:
\begin{equation}\label{eq:CPC} \tag{CPC}
\aFr(k+1,\ell) \, \aFr(k,\ell+1) \ \ge \ \aFr(k,\ell) \, \aFr(k+1,\ell+1) \ts.
\end{equation}
\end{conj}

\smallskip

The CPC was proved in \cite[Thm~3.2]{BFT} for \ts $k=\ell=1$, and in
\cite[Thm~1.4]{CPP1} for posets of width two.  We also show in  \cite[$\S$3]{CPP1},
that both the \defng{Kahn--Saks} \ts and the \defng{Graham--Yao--Yao inequalities} \ts
follow from \eqref{eq:CPC}.
%

\smallskip

\begin{thm}[{\rm Main theorem}{}] \label{t:main}
Let \ts $P=(X,\prec)$ \ts be a poset on \ts $|X|=n$ elements.
Fix distinct elements \ts $x,y,z\in X$.
Suppose that \. $\aFr(k,\ell+2)\.\aFr(k+2,\ell) > 0$. Then:
\begin{equation}\label{eq:main-thm-1}
{\aFr(k+1,\ell) \, \aFr(k,\ell+1)} \, \geq \,  \left( \tfrac12 \, + \, \tfrac{1}{4\ts n \ts \sqrt{k\ts \ell}}\right) \. \aFr(k,\ell) \, \aFr(k+1,\ell+1).
\end{equation}
Suppose that \. $\aFr(k,\ell+2)=0$ \. and \. $\aFr(k+2,\ell) > 0$.  Then:
\begin{equation}\label{eq:main-thm-2}
{\aFr(k+1,\ell)\, \aFr(k,\ell+1)} \, \geq \, \, \left( \tfrac12  \, + \,  \tfrac{1}{16 \ts n \ts k \ts \ell^2} \right) \. \aFr(k,\ell) \, \aFr(k+1,\ell+1).
\end{equation}
Suppose that \. $\aFr(k+2,\ell)=0$ \. and \. $\aFr(k,\ell+2) > 0$.  Then:
\begin{equation}\label{eq:main-thm-2-swap}
	{\aFr(k+1,\ell)\, \aFr(k,\ell+1)} \, \geq \, \left( \tfrac12  \, + \,  \tfrac{1}{16 \ts n \ts k^2 \ts \ell} \right) \. \aFr(k,\ell) \, \aFr(k+1,\ell+1).
\end{equation}
Finally, suppose that \. $\aFr(k,\ell+2) \ts = \ts \aFr(k+2,\ell) = 0$ \. and \. $\aFr(k,\ell)\. \aFr(k+1,\ell+1)>0$.  Then:
\begin{equation}\label{eq:main-thm-3}
{\aFr(k+1,\ell)\, \aFr(k,\ell+1)} \, = \, {\aFr(k,\ell)\,\aFr(k+1,\ell+1)} \, .
\end{equation}
\end{thm}

\smallskip

When \. $\aF(k,\ell)\. \aF(k+1,\ell+1)=0$, the inequality \eqref{eq:CPC}
holds trivially. Curiously, the equality \eqref{eq:main-thm-3} does not
hold in that case since the LHS can be strictly positive
(Example~\ref{ex:vanish-unequal}). Except for the natural
symmetry between \eqref{eq:main-thm-2-swap} and \eqref{eq:main-thm-2}, the proof of remaining
three cases are quite different and occupies much of the paper.

Note that computing the number \. $e(P)$ \. of linear extensions of~$P$ is $\SP$-complete \cite{BW},
even for posets of height two or dimension two~\cite{DP}.  Still, the vanishing assumptions which
distinguish the cases in the Main Theorem~\ref{t:main}, can be decided in polynomial time
(see Theorem~\ref{thm:vanishing}).

\smallskip

The proof of the Main Theorem~\ref{t:main} is a combination of geometric and combinatorial
arguments.  The former are fairly standard in the area, and used largely as a black box.
The combinatorial part is where the paper becomes technical, as the translation of
geometric ratios into the language of posets (following Stanley's pioneering
approach in~\cite{Sta}) leads to bounds on ratios of linear extensions that
have not been investigated until now.  Here we employ the \defng{combinatorics of words} \ts
technology following our previous work \cite{CPP1,CPP,CPP2} (cf.~$\S$\ref{ss:finrem-pin})

Let us emphasize that getting an explicit constant above $1/2$ in the RHS is the main
difficulty in the proof, as the $1/2$ constant is relatively straightforward to obtain
from Favard's inequality.  This was noticed independently
by Yair Shenfeld who derived it from Theorem~\ref{thm:Favard}
in the same way we did in the proof of Theorem~\ref{thm:CPC-half}.\footnote{Yair Shenfeld, personal communication (May~2,~2021).}
In another independent development, Julius Ross, Hendrik  S\"uss and Thomas Wannerer gave a
proof of the same $1/2$ lower bound using the technology of
\defng{Lorentzian polynomials} \cite{BH} combined with a technical
result from \cite{BLP}.\footnote{Julius Ross, personal communication (May~31, 2023).}

\smallskip

Our combinatorial tools also allow us to inch closer to the CPC for two classes of posets.
Fix a subset \. $A\subseteq X$.
We say that a poset $P=(X,\prec)$ is \defn{$t$-thin with respect to~$A$}, if for every \.
$u \in X\sm A$ \. there are at most \ts $t$ \ts
elements incomparable to~$\ts u$.  For $A=\emp$, such posets are a subclass of posets
of width~$t$.
Similarly, we say that a poset \ts $P=(X,\prec)$ \ts is \defn{$t$-flat with respect to~$A$},
if for every \. $u \in A$ \. there are at most \ts $t$ \ts elements comparable to~$u$.
For $A=X$, such posets are a subclass of posets of height~$t$.

\smallskip

\begin{thm} \label{t:thin}
Let \ts $P=(X,\prec)$ \ts be a finite poset.  Fix distinct elements \. $x,y,z\in X$,
and let \. $A:=\{x,y,z\}$.  Suppose that \ts $P$ \ts is either \ts $t$-thin with respect
to~$A$, or $t$-flat with respect to~$A$.  Then:
\begin{equation}\label{eq:thin-thm}
{\aFr(k+1,\ell) \, \aFr(k,\ell+1)} \, \geq \, \left(\tfrac12 \. + \. \tfrac{1}{16 \ts t \ts (t+1)^3}\right) \. {\aFr(k,\ell) \, \aFr(k+1,\ell+1)}.
\end{equation}
\end{thm}

\smallskip

Note that the constant in the RHS of \eqref{eq:thin-thm} depends only on~$t$, and
thus holds for posets of arbitrary large size~$n$, see also~$\S$\ref{ss:finrem-needle}.
We also have the following counterpart to the CPC.

\smallskip

\begin{thm}[{\rm \defng{Converse cross--product inequality}}{}] \label{t:main-converse}
Suppose that \. $\aFr(k,\ell) \. \aFr(k+1,\ell+1)>0$. Then:
\[
\aFr(k+1,\ell) \, \aFr(k,\ell+1) \, \leq  \,  2k\ell(\min\{k,\ell\}+1) \ts n \.\cdot
\. \aFr(k,\ell) \. \aFr(k+1,\ell+1).
\]
\end{thm}

\smallskip

Note that the inequality in the theorem is asymptotically tight, see Proposition~\ref{prop:converse}.
On the other hand, originally we believed in the following stronger version of the CPC:

\smallskip

\begin{conj}[{\rm\defng{Generalized cross--product conjecture}~\cite[Conj.~3.2]{CPP1}}{}]\label{conj:GCPC}
We have:
\begin{equation}\label{eq:GCPC} \tag{GCPC}
\aFr(k,\ell) \, \aFr(p,q) \,  \ \le \ \aFr(p,\ell) \, \aFr(k,q)  \quad \text{for all} \quad k\le p\ts, \, \ell \le q\ts.
\end{equation}
\end{conj}

\smallskip

For \. $p=k+1$ \. and \. $q=\ell+1$, where \. $k,\ell\geq 1$, this gives \eqref{eq:CPC}.
In \cite[Thm.~3.3]{CPP1}, the inequality \eqref{eq:GCPC} was proved for posets of width two.
However, here we show that it fails in full generality:

\smallskip

\begin{thm} \label{t:GCPC-false}
The inequality \eqref{eq:GCPC} fails for an infinite family of posets of width three.
\end{thm}

\smallskip

Our final result further confirms that CPC is somehow special among similar families of
inequalities. While these other inequalities are not always true, 
they are not simultaneously too far off in the following sense.

\smallskip

\begin{thm} \label{t:CPC-two-three}
For every \ts $P=(X,\prec)$, every distinct \ts $x,y,z\in X$, and every \ts $k,\ell\ge 1$, at least two of the inequalities \ts \eqref{eq:CPC}, \ts
\eqref{eq:CPC-2} \ts and \ts \eqref{eq:CPC-3} \ts are true, where
	\begin{align}\label{eq:CPC-2}\tag{CPC1}
		\aFr(k+2,\ell) \. \aFr(k,\ell+1) \ &\leq \   \aFr(k+1,\ell) \. \aFr(k+1,\ell+1),\\
		\label{eq:CPC-3}\tag{CPC2}
		\aFr(k,\ell+2) \. \aFr(k+1,\ell) \ &\leq \   \aFr(k,\ell+1) \. \aFr(k+1,\ell+1).
	\end{align}
\end{thm}

\smallskip

We prove that inequalities \eqref{eq:CPC-2} \ts and \ts \eqref{eq:CPC-3} \ts
hold for posets of width two (Corollary~\ref{c:CPC23-width-two}).  However, they
are false on infinite families of counterexamples (Proposition~\ref{p:CPC3-false}).
By Theorem~\ref{t:CPC-two-three}, this means that the CPC
holds in all these cases.

\smallskip

\subsection*{Paper structure}  We start with a short background Section~\ref{sec:CPC} on
mixed volumes and variations on the Alexandrov--Fenchel inequalities.  This section is
self-contained in presentation, and uses several well-known results as a black box.
In a lengthy Section~\ref{sec:poset} we show how cross product inequalities arise as
mixed volume, and make some useful calculations.  We also prove Theorem~\ref{t:CPC-two-three}.

We begin our combinatorial study of linear extensions in Section~\ref{s:vanish}, where
we give explicit conditions for vanishing of \. $\aF(k,\ell)$, and explore the consequences
which include the equality~\eqref{eq:main-thm-3}.
In Sections~\ref{s:words} and~\ref{s:cross-vanish}, we prove different cross product inequalities
in the nonvanishing and vanishing case, respectively.  We conclude with explicit examples
(Section~\ref{s:explicit}) and final remarks (Section~\ref{s:finrem}).


\bigskip

\section{Mixed volume inequalities}
\label{sec:CPC}


\subsection{Alexandrov--Fenchel inequalities}\label{ss:mixed-volumes}
  Fix \ts $n \geq 1$.
For two sets \. $A, B \subset \Rb^n$ \. and constants \ts $a,b>0$, denote by
\[ aA+bB \ := \ \bigl\{ \ts a\xb+ b\yb  \, : \, \xb \in A, \yb \in B  \ts \bigr\}
\]
the \defnb{Minkowski sum} of these sets.
For a  convex body \ts $\iA \subset \Rb^n$ \ts with affine dimension $d$, denote by \ts $\Vol_d(\iA)$ \ts the
volume of~$\iA$.
One of the basic result in convex geometry is \defng{Minkowski's theorem}
that the volume of convex bodies with affine dimension $d$ behaves as a homogeneous polynomial
of degree~$d$ with nonnegative coefficients:

\smallskip

\begin{thm}[{\rm Minkowski, see e.g.~\cite[$\S$19.1]{BZ-book}}{}]\label{thm:Minkowski}
	For all convex bodies \. $\iA_1, \ldots, \iA_r \subset \Rb^n$ \. and \. $\lambda_1,\ldots, \lambda_r > 0$,
	we have:
	\begin{equation}\label{eq:mixed volume definition}
		\Vol_d(\lambda_1 \iA_1+ \ldots + \lambda_r \iA_r) \ =  \ \sum_{1 \ts \le \ts i_1\ts ,\ts \ldots \ts , \ts i_d\ts \le \ts r} \. \iV\bigl(\iA_{i_1},\ldots, \iA_{i_d}\bigr) \, \lambda_{i_1} \ts\cdots\ts \lambda_{i_d}\.,
	\end{equation}
	where the functions \ts $\iV(\cdot)$ \ts are nonnegative and symmetric, and where $d$ is the affine dimension of \ts $\lambda_1 \iA_1+ \ldots + \lambda_r \iA_r$ \ts $($which does not depend on the choice of \ts $\lambda_1,\ldots, \lambda_r)$.
\end{thm}

\medskip

The coefficients \. $\iV(\iA_{i_1},\ldots, \iA_{i_d})$ \. are called \defnb{mixed volumes}
of \. $\iA_{i_1}, \ldots, \iA_{i_d}$\..
We use \ts $d:=d(\iA_1,\ldots, \iA_r)$ \ts to denote  the affine dimension of the Minkowski sum \. $\iA_1+\ldots +\iA_r$\..

There are many classical inequalities concerning mixed volumes,
and here we list those that will be used in this paper.
Let \. $\iA,\iB,\iC$, $\iQ_1,\ldots, \iQ_{d-2}$ \. be convex bodies in
$\Rb^n$\.. We denote \. $\ibQ=(\iQ_1,\ldots,\iQ_{d-2})$ \. and
use   \.  $\iV_{\ibQ}(\cdot, \cdot)$ \. as a shorthand for \. $\iV(\cdot, \cdot, \iQ_1, \ldots, \iQ_{d-2})$\..

\smallskip

\smallskip


\begin{thm}[{\rm \defng{Alexandrov--Fenchel inequality}, see e.g.~\cite[$\S$20]{BZ-book}}{}]\label{thm:AF}
	\begin{equation}\label{eq:AF} \tag{AF}
		\iV_{\ibQ}(\iA,\iB)^2  \, \geq \, \iV_{\ibQ}(\iA,\iA) \.  \iV_{\ibQ}(\iB,\iB).
	\end{equation}
\end{thm}

\smallskip

%

%

The following technical result generalizes Theorem~\ref{thm:AF} to inequalities involving differences
in~\eqref{eq:AF}; see e.g.\ \cite[$\S7.4$]{Sch}.

\smallskip

\begin{thm}[{\rm see e.g.~\cite[Lemma~7.4.1]{Sch}}{}]\label{thm:Sch}
	We have
	\begin{equation}\label{eq:squares}
	\begin{split}
		& \big(\iV_{\ibQ}(\iA,\iC)^2 - \iV_{\ibQ}(\iA,\iA) \. \iV_{\ibQ}(\iC,\iC) \big) \. \big( \iV_{\ibQ}(\iB,\iC)^2 - \iV_{\ibQ}(\iB,\iB) \. \iV_{\ibQ}(\iC,\iC)  \big)  \\
	& \hskip1.cm \geq  \ \big(\iV_{\ibQ}(\iA,\iC) \. \iV_{\ibQ}(\iB,\iC) \ - \  \iV_{\ibQ}(\iA,\iB) \. \iV_{\ibQ}(\iC,\iC) \big)^2.
	\end{split}
	\end{equation}
\end{thm}

\medskip


\subsection{Favard's inequality for the cross-ratio}\label{ss:mixed-cross}
Towards proving the Main Theorem~\ref{t:main}, we are most interested in bounds on the
\defn{cross-ratio} \.
$${\Upsilon}_{\ibQ}(\iA,\iB,\iC) \ := \ \frac{\iV_{\ibQ}(\iA, \iC) \. \iV_{\ibQ}(\iB, \iC) }{\iV_{\ibQ}(\iA, \iB) \. \iV_{\ibQ}(\iC, \iC) }\,.
$$
We start with the following well-known result which goes back to Favard (see $\S$\ref{ss:finrem-hist}).

\smallskip

\begin{thm}[{\rm \defng{Favard's inequality}, see~e.g.~\cite[Lemma~5.1]{BGL}}{}]\label{thm:Favard}
	Suppose we have $$\iV_{\ibQ}(\iA, \iB) \. \iV_{\ibQ}(\iC, \iC) >0\ts.$$
Then:
\begin{equation}
\label{eq:Favard}
\frac{\iV_{\ibQ}(\iA, \iC) \, \iV_{\ibQ}(\iB, \iC)}{\iV_{\ibQ}(\iA, \iB) \, \iV_{\ibQ}(\iC, \iC)}  \ \geq \ \frac{1}{2}\..
\end{equation}
\end{thm}

\smallskip

In the next section, we use order polytopes to write the cross product ratio in~\eqref{eq:CPC}
into the cross-ratio \ts $\Ups$.   Then Favard's inequality~\eqref{eq:Favard} \ts $\Ups \ge 1/2$ \ts
easily gives the constant \ts $1/2$ \ts in the inequalities in the Main Theorem~\ref{t:main}
(see Theorem~\ref{thm:CPC-half}).
To move beyond \ts $1/2$ \ts we need to strengthen~\eqref{eq:Favard}, see below.

\smallskip

\begin{rem}
From geometric point of view, the constant \ts $1/2$ \ts in the inequality~\eqref{eq:Favard}
is sharp.  For example, take \ts $\iA$ \ts and \ts $\iB$ \ts non-collinear line segments,
and \ts $\iC=\iA+\iB$, see e.g.\ \cite[Prop.~5.1]{AFO} and~\cite[Thm~6.1]{SZ}.
However, for various families of convex bodies, it is possible to improve the
constant perhaps, although not to~1 as one would wish.  For example, when \ts $\iC$ \ts
is a unit ball in~$\rr^2$ the constant can be improved to \ts $2/\pi$ \ts \cite[Prop.~5.3]{AFO}.
\end{rem}

\medskip

\subsection{Better cross-ratio inequalities}\label{ss:mixed-more}
The following two results follow from \eqref{eq:squares} by elementary arguments.
They are variations on inequalities that are already known in the literature.
We include simple proofs for completeness.

\smallskip

\begin{prop}\label{prop:Vol-CPC-1}
Suppose that  \. $\iV_{\ibQ}(\iA, \iB) \. \iV_{\ibQ}(\iC, \iC) >0$\ts. Then:	
\begin{equation}\label{eq:Vol-CPC-1}
\frac{\iV_{\ibQ}(\iA, \iC) \. \iV_{\ibQ}(\iB, \iC)}{\iV_{\ibQ}(\iA, \iB) \iV_{\ibQ}(\iC, \iC)}  \ \geq \ \frac{1}{2} \bigg(1 + \frac{\sqrt{\iV_{\ibQ}(\iA, \iA) \. \iV_{\ibQ}(\iB, \iB)}}{\iV_{\ibQ}(\iA, \iB)} \bigg).
	\end{equation}
\end{prop}

\smallskip

\begin{proof}
	Let $\alpha_1,\alpha_2,\beta_1,\beta_2$  be nonnegative real numbers given by
	\begin{alignat*}{2}
	\alpha_1 \, &:= \, \frac{{\iV_{\ibQ}(\iA, \iC)}}{\sqrt{\iV_{\ibQ}(\iA, \iB) \. \iV_{\ibQ}(\iC, \iC)}}\, , \qquad
	&& \alpha_2 \, := \, \frac{{\iV_{\ibQ}(\iB, \iC)}}{\sqrt{\iV_{\ibQ}(\iA, \iB) \. \iV_{\ibQ}(\iC, \iC)}}\, , \\
	\beta_1 \, &:= \, \frac{\iV_{\ibQ}(\iA,\iA)}{\iV_{\ibQ}(\iA,\iB)}\,, \qquad
	&& \beta_2 \, := \, \frac{\iV_{\ibQ}(\iB,\iB)}{\iV_{\ibQ}(\iA,\iB)}\,.
	\end{alignat*}
Note that \. $\beta_1 \beta_2 \leq 1$ \. by~\eqref{eq:AF}.
	By perturbing the convex bodies again if necessary, we can without loss of generality assume that \ts $\beta_1 \beta_2 < 1$.

	In this notation, we can rewrite \eqref{eq:squares} as
	\[  ( \alpha_1 \alpha_2-1)^2  \ \leq \  (\alpha_1^2-\beta_1) \. (\alpha_2^2-\beta_2).  \]
	Rearranging the terms, this gives:
	\begin{equation}\label{eq:ab-1}
	\alpha_1\alpha_2  \ \geq \   \frac{1}{2} \, + \, \frac{1}{2} \. \big(\alpha_1^2 \beta_2 + \alpha_2^2 \beta_1 \big) \,  - \, \frac{1}{2} \.  \beta_1 \beta_2\..
	\end{equation}
	By applying the AM--GM inequality to the terms \. $\big(\alpha_1^2 \beta_2 + \alpha_2^2 \beta_1 \big)$\., we get
	\[ \alpha_1 \. \alpha_2 \ \geq \ \frac{1}{2} \, + \,  \alpha_1 \alpha_2 \sqrt{\beta_1\. \beta_2}  \,  - \, \frac{1}{2} \. \beta_1 \beta_2\..  \]
	Rearranging the terms, this gives:
	\[ \big(1 \. - \. \sqrt{\beta_1\beta_2} \big) \. \alpha_1 \alpha_2  \ \geq \  \frac{1}{2} \big(1-\beta_1\beta_2 \big).  \]
	Since $\beta_1\beta_2< 1$, we can divide both side of the inequality above by \. $\big(1 \. - \. \sqrt{\beta_1\beta_2} \big)$ \. and get
		\[  \. \alpha_1 \alpha_2  \ \geq \  \frac{1}{2} \big(1+\sqrt{\beta_1\beta_2} \big).  \]
	This gives the desired \eqref{eq:Vol-CPC-1}.
\end{proof}

\smallskip


We now present a variant of Proposition~\ref{prop:Vol-CPC-1} in a degenerate case.

\smallskip

\begin{prop}
\label{prop:Vol-CPC-2}
Suppose that  \. $\iV_{\ibQ}(\iA, \iB) \iV_{\ibQ}(\iC, \iC) >0$ \. and \. $\iV_{\ibQ}(\iB,\iB)=0$. Then:
\begin{equation}\label{eq:Vol-CPC-2}
	\frac{\iV_{\ibQ}(\iA, \iC) \. \iV_{\ibQ}(\iB, \iC)}{\iV_{\ibQ}(\iA, \iB) \. \iV_{\ibQ}(\iC, \iC)} \ \geq \    \bigg(1+ \sqrt{1-\frac{\iV_{\ibQ}(\iA, \iA) \. \iV_{\ibQ}(\iC, \iC)}{\iV_{\ibQ}(\iA, \iC)^2}} \bigg)^{-1}.
\end{equation}
\end{prop}

\smallskip

\begin{proof}
First note that \eqref{eq:squares} gives:
	\begin{equation}\label{eq:sqrt-1}
	\begin{split}
		& \big(\iV_{\ibQ}(\iA,\iC) \. \iV_{\ibQ}(\iB,\iC) \ - \  \iV_{\ibQ}(\iA,\iB) \. \iV_{\ibQ}(\iC,\iC) \big)^2  \  \\
		& \qquad \leq  \ \big(\iV_{\ibQ}(\iA,\iC)^2 - \iV_{\ibQ}(\iA,\iA) \. \iV_{\ibQ}(\iC,\iC) \big) \.  \iV_{\ibQ}(\iB,\iC)^2.
	\end{split}
\end{equation}
We assume  without loss of generality  that
\begin{equation}\label{eq:sqrt-2}
 \iV_{\ibQ}(\iA,\iC) \. \iV_{\ibQ}(\iB,\iC) \ <  \   \iV_{\ibQ}(\iA,\iB) \. \iV_{\ibQ}(\iC,\iC)\ts.
\end{equation}
In fact, otherwise, since the right side of \eqref{eq:Vol-CPC-2} is at most \ts $1$ \ts we immediately have
\eqref{eq:Vol-CPC-2}.

Now note that \. $ \iV_{\ibQ}(\iA,\iC) \. \iV_{\ibQ}(\iB,\iC) > 0$ \. by \eqref{eq:Favard}
and by the assumption of the theorem.  Taking the square root of \eqref{eq:sqrt-1} using~\eqref{eq:sqrt-2}, and then
dividing by \. $ \iV_{\ibQ}(\iA,\iC) \. \iV_{\ibQ}(\iB,\iC)$\ts, we get:
\[ \frac{\iV_{\ibQ}(\iA,\iB) \. \iV_{\ibQ}(\iC,\iC)}{\iV_{\ibQ}(\iA,\iC) \. \iV_{\ibQ}(\iB,\iC)}  \ - \  1  \ \leq \   \sqrt{1-\frac{\iV_{\ibQ}(\iA, \iA) \. \iV_{\ibQ}(\iC, \iC)}{\iV_{\ibQ}(\iA, \iC)^2}}\,.  \]
This is equivalent to \eqref{eq:Vol-CPC-2}.
\end{proof}

\bigskip

\section{Poset inequalities via mixed volumes}\label{sec:poset}

\subsection{Definitions and notation}\label{ss:CPC-def}
We refer to~\cite{Tro} for some standard posets notation.
Let \. $P=(X, \prec)$ \. be a poset with \ts $|X|=n$ \ts elements.
A \defn{dual poset} \ts is a poset \ts $P^\ast=(X,\prec^\ast)$, where
\ts $x\prec^\ast y$ \ts if and only if \ts $y \prec x$.

We somewhat change the notation and
fix distinct elements \. $z_1,z_2,z_3\in X$ \. which we use throughout the paper.
As in the introduction, for \ts $k,\ell \geq 1$ \ts let
\[ \cF(k,\ell) \ := \ \{L \in \Ec(P) \, : \, L(z_2)-L(z_1)=k, \.  L(z_3)-L(z_2)=\ell \}, \]
and let \. $\aF(k,\ell) \ts := \ts \big|\cF(k,\ell)\big|$.
We will write \ts $\aF_{P,z_1,z_2,z_3}(k,\ell)$ \ts in place of \ts $\aF(k,\ell)$ \ts
when there is a potential ambiguity in regards to the underlying poset $P$ and the elements \ts $z_1,z_2,z_3 \in X$\ts.

\smallskip

\subsection{Half CPC}\label{ss:CPC-main}
We first prove that \eqref{eq:CPC} holds up to a factor of~$2$.  Formally, start with
the following weak version of the Main Theorem~\ref{t:main}:

\smallskip

\begin{thm}\label{thm:CPC-half}
		For every \ts $k,\ell \geq 1$\ts, we have:
		\begin{equation}\label{eq:biCPC-1}\tag{half-CPC}
	 \aFr(k,\ell) \. \aFr(k+1,\ell+1) \ \leq \  2 \. \aFr(k+1,\ell) \. \aFr(k,\ell+1).
		\end{equation}
\end{thm}

\smallskip

To prove Theorem~\ref{thm:CPC-half}, we will first interpret the quantity \. $\aF(k,\ell)$ \.
as in the language of mixed volumes.  Here we follow Stanley's approach in \cite{Sta} (see also~\cite{KS}).

\smallskip

Fix a poset \ts $P=(X,\prec)$, and let \ts $\Rb^X$ \ts be the space of real vectors  \ts $\vb$ \ts
that are indexed by elements \ts $x\in X$.
Throughout this section, the entries of the vector \ts $\vb$ \ts that corresponds to $x \in X$ will be denoted by $\vb(x)$, to maintain legibility when $x$ are substituted with elements \ts $z_i$\ts.
The \defnb{order polytope}  \. $\iK:=\iK(P) \subset \Rb^X$ \. is defined as follows:
\[ \iK \ := \ \big\{\ts\vb \in \Rb^X \ \, : \, \ \vb(x) \leq  \vb (y) \, \text{ for all } \, x \prec y, \, x, y \in X\., \ \.  \text{ and } \ \.  0 \leq \vb(x)\leq 1 \, \text{ for all } \, x \in X  \ts\big\}.  \]
Let  \. $\iK_1$, \ts $\iK_2$, \ts $\iK_3 \subseteq \iK$ \. be the slices of the order polytope defined as follows:
\begin{equation}\label{eq:K-polytope}
	\begin{split}
		\iK_1 \ &:= \ \{\. \vb \in \iK \, : \,   \vb(z_2) - \vb(z_1) \. = \. 1,  \  \vb(z_3)  -  \vb(z_2) \. = \. 0   \.  \}\ts,\\
		\iK_2 \ &:= \ \{\. \vb \in \iK \, : \,  \vb(z_2) - \vb(z_1) \. = \. 0,  \  \vb(z_3)  -  \vb(z_2) \. = \. 1  \.  \}\ts,\\
		\iK_3 \ &:= \ \{\. \vb \in \iK \, : \, \vb(z_2) - \vb(z_1) \. = \. \vb(z_3)   -  \vb(z_2) \. = \. 0  \.  \}\ts.
	\end{split}
\end{equation}
Note that all Minkowski sums of these three polytopes have affine dimension $d=n-2$.
\smallskip

\begin{lemma}\label{lem:CPC-Fkl-mixed-volume}
	Let \ts $k,\ell \geq 1$\ts, \ts $k+\ell \le n$.  We have:
	\begin{equation}\label{eq:CPC-Fkl-mixed-volume}
		\aFr(k,\ell) \ = \ (n-2)! \  \iV(\underbrace{\iK_1,\ldots, \iK_1}_{k-1}, \underbrace{\iK_2,\ldots,\iK_2}_{\ell-1}, \underbrace{\iK_3,\ldots, \iK_3}_{n-k-\ell}).
	\end{equation}
\end{lemma}

\smallskip

This lemma follows by a variation on  the argument in the proof of \cite[Thm~3.2]{Sta} and
\cite[Thm~2.5]{KS}.

\smallskip

\begin{proof}
	For \. $0 < s,t < 1$, \. $0 < s+t < 1$, define
	\[ \iK^{(s,t)} \ := \ \big\{\. \vb \in \iK \, : \,  \vb(z_2) - \vb(z_1) \. = \. s,  \  \vb(z_3)   -  \vb(z_2) \. = \. t  \.  \big\}.
\]
	Note that \. $\iK^{(s,t)} \. = \.  s \ts \iK_1 \. + \.  t \ts \iK_2  \. + \.   (1-s-t) \ts \iK_3$\..
%
	Let us now compute the volume of \ts $\iK^{(s,t)}$\ts.

	For every \ts $L \in \Ec(P)$ \ts we denote by \ts $\Delta_{L} \subset \iK^{(s,t)}$ \ts the polytope
	\[ \Delta_L \ := \  \{\. \vb \in  \iK^{(s,t)} \. \mid \. \vb(x) \. \leq \. \vb (y) \ \text{ whenever } \ L(x) \leq  L(y) \. \}.  \]
	Note that \ts $\iK^{(s,t)}$ \ts is the union of \ts $\Delta_L$'s \ts over all linear extensions $L$ such that \ts $L(z_1)<L(z_2)<L(z_3)$, and furthermore all \ts $\Delta_L$'s \ts have pairwise disjoint interiors.
	Hence it remains to compute the volume of $\Delta_L$'s.
	
	Let \ts $L \in \cF(k,\ell)$ \ts for some \ts $k,\ell \geq 1$\ts, \.  let \ts $h:=L(z_1)$\ts, and let \ts $x_i$ \ts ($i\in \{1,\ldots,n\}$) \ts  be the $i$-th smallest element under the total order of $L$.
	Note that \ts $z_1=x_h$, \ts $z_2=x_{h+k}$\ts, \ts and \ts $z_3=x_{h+k+\ell}$\ts.
	Then    \ts $\Delta_L$ \ts consists of \ts $\vb \in \Rb^X$ \ts that satisfies these three inequalities: \ts $0 \leq \vb(x_1) \leq \vb(x_2) \leq \ldots \leq \vb(x_n) \leq 1$, \.  $\vb(x_{h+k})=\vb(x_{h})+s$\., \. $\vb(x_{h+k+\ell})=\vb(x_{h})+s+t$.
Denote by \ts $\Phi: \Rb^X\to \Rb^X$ \ts the (volume preserving) transformation defined
as follows:
\. $\Phi(\vb) = \wb$, where
	\begin{alignat*}{2}
		&\wb(x_i) \ = \ \vb(x_i) \quad  && \text{ if } \quad i \. \leq \. h,\\
		&\wb(x_i) \ = \ \vb(x_i)-\vb(x_h) \quad  && \text{ if } \quad h \. < \. i \. \leq \. h+k,\\
		&\wb(x_i) \ = \ \vb(x_i)-\vb(x_h)-s \quad  && \text{ if } \quad h+k \. < \. i \. \leq \. h+k+\ell,\\
		&\wb(x_i) \ = \ \vb(x_i)-s-t \quad  && \text{ if } \quad h+k+\ell \. < \. i  \. \leq \. n\ts.
		\end{alignat*}
	Then the image \. $\Phi(\Delta_L)$ \. is the set of \. $\wb \in \Rb^X$ \. that satisfies
	\begin{alignat*}{2}
		&0 \. \leq \. \wb(x_1) \. \leq \. \ldots \. \leq \. \wb(x_h) \. \leq \. \wb(x_{h+k+\ell+1}) \. \leq \. \ldots \. \leq \wb({x_n}) \. \leq 1-s-t\ts, \\
		& 0 \. \leq \. \wb(x_{h+1}) \. \leq \. \ldots \. \leq \.  \wb(x_{h+k}) \. = \. s\ts, \quad \text{and}\\
		& 0 \. \leq \. \wb(x_{h+k+1}) \. \leq \. \ldots \. \leq \.  \wb(x_{h+k+\ell}) \. = \. t\ts.
	\end{alignat*}
	This set is the direct product of three simplices and has volume
	\[ \rho(s,t) \, := \, \frac{s^{k-1}}{(k-1)!} \. \times \. \frac{t^{\ell-1}}{(\ell-1)!} \. \times \. \frac{(1-s-t)^{n-k-\ell}}{(n-k-\ell)!}\,. \]
	It follows from here that
	\begin{align*}
		&\Vol_{d}\big(\iK^{(s,t)}\big) \ = \ \sum_{k,\ell \geq 1} \, \sum_{L \ts\in\ts \cF(k,\ell)} \Vol_d(\Delta_L) \ = \ \sum_{k,\ell \geq 1} \. \sum_{L \in \cF(k,\ell)}  \. \rho(s,t)  \. \\
		&\qquad = \ \sum_{k,\ell \geq 1} \binom{n-2}{n-k-\ell, \. k-1, \. \ell-1}
		\, \frac{\aF(k,\ell)}{(n-2)!} \,\. s^{k-1} \. t^{\ell-1} \. (1-s-t)^{n-k-\ell}.
	\end{align*}
 Since the choice of $s,t$ is arbitrary, equation \eqref{eq:CPC-Fkl-mixed-volume}  follows
 from the Minkowski Theorem~\ref{thm:Minkowski}. \end{proof}

\smallskip

\begin{proof}[Proof of Theorem~\ref{thm:CPC-half}]
	Let $d=n-2$, and let \. $\iA,\iB,\iC$, $\iQ_1,\ldots, \iQ_{d-2} \subset \iK$ \. be given by
	\begin{equation}\label{eq:charlie1}
	\begin{split}
		& \iA \, \gets \, \iK_1, \quad \iB \, \gets \, \iK_2, \quad \iC \, \gets \, \iK_3, \quad \text{and}\\
		& \iQ_1,\ldots, \iQ_{d-2} \, \gets \,  \underbrace{\iK_1,\ldots, \iK_1}_{k-1}, \underbrace{\iK_2,\ldots,\iK_2}_{\ell-1}, \underbrace{\iK_3,\ldots, \iK_3}_{n-k-\ell}.
	\end{split}
	\end{equation}
	The theorem now follows by applying Lemma~\ref{lem:CPC-Fkl-mixed-volume} into Theorem~\ref{thm:Favard}.
\end{proof}

\medskip

\subsection{Applications to cross products} \label{ss:CPC-app}
We now quickly derive the key applications of mixed volume cross-ratio inequalities
for the cross product inequalities.

\smallskip

\begin{prop}\label{prop:cpc1+eps}
Suppose that \. $\aFr(k,\ell)\. \aFr(k+1,\ell+1)>0$. Then:
\begin{align*}
\frac{ \aFr(k+1,\ell) \. \aFr(k,\ell+1) }{\aFr(k,\ell) \. \aFr(k+1,\ell+1)}
\ \geq \  \frac{1}{2} \ + \  \frac{\sqrt{\aFr(k,\ell+2)\. \aFr(k+2,\ell)} }{2 \. \aFr(k+1,\ell+1)}\,.
\end{align*}\end{prop}

\smallskip

\begin{proof}
	Let $d=n-2$, and let \. $\iA,\iB,\iC$, $\iQ_1,\ldots, \iQ_{d-2} \subset \iK$ \. be given by \eqref{eq:charlie1}.
	The conclusion of the proposition now follows from Lemma~\ref{lem:CPC-Fkl-mixed-volume} and Proposition~\ref{prop:Vol-CPC-1}.
\end{proof}

\smallskip

\begin{prop}\label{p:cpc+eps0}
Suppose that \. $\aFr(k,\ell)\. \aFr(k+1,\ell+1)>0$ \. and \. $\aFr(k,\ell+2)=0$. Then:
\begin{align*}
\frac{ \aFr(k+1,\ell)\aFr(k,\ell+1) }{\aFr(k+1,\ell+1)\aFr(k,\ell)} \ \geq \  \bigg(1 + \sqrt{ 1 - \frac{ \aFr(k,\ell) \aFr(k+2,\ell)}{\aFr(k+1,\ell)^2} } \. \bigg)^{-1}.
\end{align*}
\end{prop}

\smallskip
\begin{proof}
	Let $d=n-2$, and let \. $\iA,\iB,\iC$, $\iQ_1,\ldots, \iQ_{d-2} \subset \iK$ \. be given by \eqref{eq:charlie1}.
	The conclusion of the proposition now follows from Lemma~\ref{lem:CPC-Fkl-mixed-volume} and Proposition~\ref{prop:Vol-CPC-2}.
\end{proof}

\medskip

\subsection{More half-CPC inequalities}\label{ss:CPC-other}
We start with the following half-versions of \eqref{eq:CPC-2} and  \eqref{eq:CPC-3}.
The proofs follow the proof of Theorem~\ref{thm:CPC-half} given above.

\smallskip

\begin{lemma}\label{lem:CPC-other-half}
			For every \ts $k,\ell \geq 1$, we have:
	\begin{align}\label{eq:biCPC-2}\tag{half-CPC1}
		\aFr(k+2,\ell) \. \aFr(k,\ell+1) \ &\leq \  2 \. \aFr(k+1,\ell) \. \aFr(k+1,\ell+1),\\
		  \label{eq:biCPC-3}\tag{half-CPC2}
		  \aFr(k,\ell+2) \. \aFr(k+1,\ell) \ &\leq \  2 \. \aFr(k,\ell+1) \. \aFr(k+1,\ell+1).
	\end{align}
\end{lemma}

\smallskip

\begin{proof}
We again let $d=n-2$ and let  \. $\iQ_1,\ldots, \iQ_{d-2} \subset \iK$ \. be given by \eqref{eq:charlie1}.
Then
	\eqref{eq:biCPC-2} follows by applying  Lemma~\ref{lem:CPC-Fkl-mixed-volume} into Theorem~\ref{thm:Favard}, with the choice
	\[ \iA \.\gets \.  \iK_3\., \quad \iB \.\gets \.\iK_2 \quad \text{and} \quad \iC \.\gets \. \iK_1\..\]
	Similarly, \eqref{eq:biCPC-3} follows from the choice
		\[ \iA \.\gets \.  \iK_3\., \quad \iB \.\gets \. \iK_1 \quad \text{and} \quad \iC \.\gets \. \iK_2\..\]
	This completes the proof.
\end{proof}

\smallskip

Note that  \eqref{eq:CPC-2} is a dual inequality to  \eqref{eq:CPC-3} in the following sense.
Let \. $P^\ast:=(X,\prec^\ast)$ \. be the \defnb{dual poset} of~$P$,
i.e. \. $x\prec^\ast y$ \. if and only if \. $x\succ y$\..
Let \. $z_1^\ast:=z_3$, \. $z_2^\ast:=z_2$, \. $z_3^\ast:=z_1$.
Then \. $\aF_{P,z_1,z_2,z_3}(k,\ell) = \aF_{P^\ast, z_1^\ast,z_2^\ast,z_3^\ast}(\ell,k)$ \.
by the maps that send \ts linear extensions of $P$ \ts  to \ts linear extensions of $P^\ast$ \ts by reversing the total order.

On the other hand,
one can think of \eqref{eq:CPC-2} and \eqref{eq:CPC-3} as negative  variants of \eqref{eq:CPC}, in the following sense.
Let \ts $z_1':=z_2$, \ts $z_2':=z_1$, \ts $z_3':=z_3$, and we write \ts $\aF= \aF_{P,z_1,z_2,z_3}$ \ts and \ts \ts $\aF'= \aF_{P,z_1',z_2',z_3'}$.
Then, for every integer \ts $k,\ell$,
\begin{align*}
	\aF(k,\ell) \ &= \ \big|\{L \in \Ec(P) \, : \, L(z_2)-L(z_1)=k, \.  L(z_3)-L(z_2)=\ell \}\big|\\
	\ &= \  \big|\{L \in \Ec(P) \, : \, L(z_1)-L(z_2)=-k, \.  L(z_3)-L(z_1)=\ell+k \}\big|\\
	\ &= \ \aF'(-k, \ell+ k).
\end{align*}
Let \ts $k':=-k-1$ \ts and \ts $\ell':=\ell+k$. Under this change of variable, \eqref{eq:CPC}  then becomes
\[ \aF'(k'+1,\ell') \. \aF'(k',\ell'+2) \  \leq \ \aF'(k',\ell'+1) \. \aF'(k'+1,\ell'+1),  \]
which coincides with \eqref{eq:CPC-3} in this case.

Note, however, that \eqref{eq:CPC} does not imply \eqref{eq:CPC-2} and vice versa, since $k'$ are necessarily negative under this transformation. In fact, as mentioned in the introduction, we will present counterexamples to \eqref{eq:CPC-2} in~$\S$\ref{ss:explicit-gen-CPC}.
\medskip

\subsection{Variations on the theme}\label{ss:CPC-var}
The following three inequalities are variations on \eqref{eq:CPC}.


\smallskip

\begin{lemma}\label{lem:logC}
For every \. $k,\ell \geq 1$ \. we have:
\begin{align}
			\label{eq:LogC-1}\tag{LogC-1}
			\aFr(k+1,\ell+1)^2 \  &\geq \ \aFr(k+2,\ell) \. \aFr(k,\ell+2),\\
			\label{eq:LogC-2}\tag{LogC-2}
		\aFr(k,\ell+1)^2 \  &\geq \ \aFr(k,\ell) \. \aFr(k,\ell+2),\\
			\label{eq:LogC-3}\tag{LogC-3}
		\aFr(k+1,\ell)^2 \  &\geq \ \aFr(k,\ell) \. \aFr(k+2,\ell).
		\end{align}
\end{lemma}

\smallskip

\begin{proof}
Let $d=n-2$, and let \. $\iA,\iB,\iC$, $\iQ_1,\ldots, \iQ_{d-2} \subset \iK$ \. be given by \eqref{eq:charlie1}.
		It follows from the Alexandrov--Fenchel inequality \eqref{eq:AF} that
		\begin{align*}
								\iV_{\ibQ}(\iA,\iB)^2  \ &\geq \ \iV(\iA,\iA) \  \iV(\iB,\iB),\\
								\iV_{\ibQ}(\iB,\iC)^2  \ &\geq \ \iV(\iB,\iB) \  \iV(\iC,\iC),\\
								\iV_{\ibQ}(\iA,\iC)^2  \ &\geq \ \iV(\iA,\iA) \  \iV(\iC,\iC).
		\end{align*}
		By applying Lemma~\ref{lem:CPC-Fkl-mixed-volume}, we get the desired inequalities.
\end{proof}

\smallskip

\begin{rem}
The inequalities \eqref{eq:LogC-1}, \eqref{eq:LogC-2} and \eqref{eq:LogC-3}
can be viewed as extensions of Stanley's and Kahn--Saks inequalities, cf.~\cite{CPP1,CPP2}.
\end{rem}

\smallskip

\begin{cor}
\label{cor:CPC-logconcave-product}
Suppose that \. $\aFr(k,\ell) \. \aFr(k+1,\ell+1) >0$\ts.  Then we have:
\[  \frac{\aFr(k+1,\ell) \. \aFr(k,\ell+1)}{\aFr(k,\ell) \. \aFr(k+1,\ell+1)} \ \geq \ \frac{\aFr(k+2,\ell) \. \aFr(k,\ell+2)}{\aFr(k+1,\ell+1)^2} \,.  \]
In particular, if \eqref{eq:LogC-1} is an equality, then the inequality \eqref{eq:CPC} holds.
\end{cor}

\smallskip

\begin{proof}
Taking the product of \eqref{eq:LogC-1}, \eqref{eq:LogC-2} and \eqref{eq:LogC-3}, we have:
\[
\aF(k+1,\ell) \. \aF(k,\ell+1) \. \aF(k+1,\ell+1) \  \geq \
\aF(k,\ell) \. \aF(k+2,\ell) \. \aF(k,\ell+2).
\]
By the assumptions, this implies the result.\footnote{Alternatively,
the corollary follows immediately from Proposition~\ref{prop:cpc1+eps}. }
\end{proof}

\smallskip

\begin{proof}[Proof of Theorem~\ref{t:CPC-two-three}]
	First, assume that both \eqref{eq:CPC-2} and \eqref{eq:CPC-3} are false:
	\[ 		
\aligned & \aF(k+2,\ell) \. \aF(k,\ell+1) \ >  \   \aF(k+1,\ell) \. \aF(k+1,\ell+1) \quad \text{ and } \\
&	\aF(k,\ell+2) \. \aF(k+1,\ell) \ > \   \aF(k,\ell+1) \. \aF(k+1,\ell+1).
\endaligned \]
	Taking the product of both inequalities, we then get
	\[  \aF(k+2,\ell) \. \aF(k,\ell+2) \  > \ \aF(k+1,\ell+1)^2,  \]
	which contradicts \eqref{eq:LogC-1}.
	The proofs for the  other cases are analogous.
\end{proof}

\bigskip

\bigskip

\section{Vanishing of poset inequalities}\label{s:vanish}

\subsection{Poset parameters} \label{ss:vanish-par}
For an element \ts $x \in X$, let \.
$B(x) := \big\{y \in X \. : \. y \preccurlyeq x  \big\}$ \.
denote the \defn{lower order ideal} \ts generated by~$x$, and let \. $b(x):=|B(x)|$.
Similarly, let \.
$B^\ast(x) := \big\{y \in X \. : \. y \succcurlyeq  x  \big\}$ \.
denote the \defn{upper order ideal} \ts generated by~$x$, and let \. $b^\ast(x):=|B^\ast(x)|$.

By analogy, let \. $B(x,y) = \{ z\in X \. : \.  x \preccurlyeq z \preccurlyeq y\}$ \. be the \defn{interval} \ts
between $x$ and~$y$, and let \. $b(x,y) = |B(x,y)|$.
Without loss of generality we can always assume that \ts $z_1\prec z_2 \prec z_3$\ts, since
otherwise these relations can be added to the poset.  We then have \. $b(z_1,z_2)$, \ts $b(z_2,z_3) \geq 2$.

\smallskip

Let \. $x,y\in X$ \. be two incomparable elements in~$P$, write \. $y\ts\|\ts{}x$.  Define
$$\cL(x,y) \. := \. \big\{ \ts z\in X \, : \, z\ts\|\ts{}y\., \ z\preccurlyeq x \ts \big\} \quad \text{and}
\quad u(x,y) \. := \. |\cL(x,y)|\ts.
$$
Similarly, define
$$
\cL^*(x,y) \. := \. \big\{ \ts z \in X \, : \,  z\ts\|\ts{}y\., \ z\succcurlyeq x \ts \big\} \quad \text{and}
\quad u^*(x,y) \. := \. |\cL^*(x,y)|\ts.
$$
Finally, let
$$
t(x) \. := \.  \max \big\{ \ts u(x,y) \, : \, y\in X, \, y \ts\|\ts{} x\ts\big\} \quad \text{and} \quad
t^*(x) \. := \.  \max \big\{ \ts u^*(x,y) \, : \, y\in X, \, y\ts\|\ts{}x \ts\big\},
$$
and we define \. $t(x):=1, \. t^*(x):=1$ \. if every element $y \in X$ is comparable to $x$.
Clearly, \. $t(x) \leq b(x)$ \. and \. $t^*(x) \leq b^*(x)$, by definition.

In this notation, recall that a poset \ts $P=(X,\prec)$ \ts is \ts \defn{$t$-thin with respect to~$A$}, if for every \.
$u \in X\sm A$ \. we have \. $n-b(u)-b^*(u)\le t-1$.  Similarly, recall that a poset \ts $P=(X,\prec)$ \ts
is \ts \defn{$t$-flat with respect to~$A$}, if for every \. $u \in A$ \. we have \. $b(u)+b^*(u)\le t+1$.
Note that \. $t(u),t^*(u) \leq t$ \. in either case.

\medskip

\subsection{Vanishing conditions} \label{ss:vanish-vanish}
%
%
%
Recall the following conditions for existence of restricted linear extensions.

\smallskip

\begin{thm}[{\rm \ts \cite[Thm~1.12]{CPP}\ts}{}]\label{thm:stanley_vanish}
Let \. $P=(X,\prec)$ \. be a poset with \. $|X|=n$ \. elements, and let \. $z_1,\ldots,z_r\in X$ \.
be distinct elements such that \. $z_1 \prec z_2 \prec \cdots \prec z_r$ \ts.
Fix integers \. $1\leq a_1<a_2<\cdots<a_r \leq n$.
Then there exists a linear extension \. $L\in \Ec(P)$ \. with \. $L(z_i)=a_i$ \.
for all \. $1\le i \le r$ \, {\underline{\rm  if and only if}}
\begin{equation}\label{eq:Sta-gen-vanish}
\left\{\,\aligned
& b(z_i)\. \le \. a_i \.,  \ \, b^*(z_i)\le n-a_i+1 \ \  \text{for all} \ \ 1\.\le \. i \. \le \. r\ts, \ \, \text{and} \\
& a_j \. - \. a_i \, \geq \, b(z_i,z_j) -1  \ \  \text{for all} \ \  1\.\le \. i \. < \. j \. \le \. r\ts.
\endaligned\right.
\end{equation}
\end{thm}

\smallskip

We apply this result to determine the vanishing conditions for \. $\aF(k,\ell)$.

\smallskip

\begin{thm}\label{thm:vanishing}
Let \. $P=(X,\prec)$ \. be a poset with \. $|X|=n$ \. elements, and let \. $z_1\prec z_2\prec z_3$ \.
be distinct elements in~$X$.  Then  \. $\aFr(k,\ell) >0$ \, {\underline{\rm  if and only if}}
\begin{eqnarray*}
 b(z_1,z_2)-1\ \leq & k &\leq \ n+1 - b(z_1) - b^*(z_2), \\
b(z_2,z_3)-1  \ \leq & \ell &\leq \ n+1 - b^*(z_3) - b(z_2), \\
 b(z_1,z_3)-1\ \leq & k + \ell &\leq \ n+1 - b^*(z_3)-b(z_1).
\end{eqnarray*}
\end{thm}

Note that conditions in the theorem can be viewed as $6$ linear inequalities for \. $(k,\ell)\in \nn^2$.
These inequalities determine a convex polygon in~$\rr^2$ (see below).

\begin{proof}
We have that $\aF(k,\ell)>0$ if and only if there exists an integer $a$, such that the conditions of Theorem~\ref{thm:stanley_vanish} are satisfied for the elements $z_1 \prec z_2 \prec z_3$ with $a_1 =a, a_2=a+k, a_3=a+k+\ell$. Rewriting the inequalities we obtain the following conditions
\begin{equation*}
\aligned
& b(z_1,z_2) \leq k+1 \.,\quad b(z_2,z_3) \leq \ell+1 \., \quad b(z_1,z_3) \leq k +\ell+1 \qquad \text{ and }\\
& \max\{b(z_1), b(z_2)-k,b(z_3)-k-\ell\} \leq a \leq n+1- \max\{b^*(z_1), k+b^*(z_2), k+\ell+b^*(z_3)\}
\endaligned
\end{equation*}
The integer $a$ exists if and only if the last inequalities are consistent, which leads to
\begin{equation*}
\aligned
& b(z_1,z_2)+1 \leq k \.,\quad b(z_2,z_3)+1 \leq \ell \., \quad b(z_1,z_3)+1 \leq k +\ell \qquad \text{ and }\\
& \max\{b(z_1), b(z_2)-k,b(z_3)-k-\ell\} + \max\{b^*(z_1), k+b^*(z_2), k+\ell+b^*(z_3)\}  \leq n+1
\endaligned
\end{equation*}
Noting that $b(z_i) +b^*(z_i) \leq n+1$ for all $i$, the second inequality translates to $6$ unconditional linear inequalities for $k$ and $\ell$, which can be written as
\begin{eqnarray*}
b(z_2)+b^*(z_1) - n-1 \ \leq & k &\leq \ n+1 - b(z_1) - b^*(z_2), \\
b^*(z_2)+b(z_3) -n-1  \ \leq & \ell &\leq \ n+1 - b^*(z_3) - b(z_2), \\
b^*(z_1)+b(z_3) -n-1 \ \leq & k + \ell &\leq \ n+1 - b^*(z_3)-b(z_1).
\end{eqnarray*}
Finally, since \. $|X|=n$, we also have:
\[ b(z_i)+b^*(z_j) - n \ \leq \  b(z_j,z_i)  \quad \text{for all \ $1\leq j<i \leq 3$\ts.}
\]
Combining with the previous inequalities, we obtain the desired conditions.
\end{proof}

\smallskip

\begin{cor}\label{c:vanish-either}
Suppose that \. $\aFr(k+1,\ell) \. \aFr(k,\ell+1)=0$. Then \. $\aFr(k,\ell)\. \aFr(k+1,\ell+1)=0$.
\end{cor}

\smallskip

\begin{proof}
Let \. $\cS:= \big\{(k,\ell)\in \nn^2 \. : \. \aF(k,\ell)>0\big\}$ \. denote the support
of \. $\aF(\cdot,\cdot)$.  By Theorem~\ref{thm:vanishing} we have \. $\cS$ \. is a (possibly degenerate)
hexagon with sides parallel to the axis and the line \. $k+\ell=0$. Observe that if \. $(k,\ell), \ts (k+1,\ell+1)\in \cS$,
then we also have \. $(k+1,\ell), \ts (k,\ell+1)\in \cS$.  In other words, if \. $\aF(k,\ell)\. \aF(k+1,\ell+1)=0$,
then we also have \. $\aF(k+1,\ell) \. \aF(k,\ell+1)=0$, as desired.
\end{proof}

\medskip

\subsection{Cross product equality in the vanishing case}\label{ss:vanish-equality}
We are now ready to prove \eqref{eq:main-thm-3} in the main theorem.

\smallskip

\begin{lemma}\label{lem:strict-CPC-zero}
Let $P=(X,\prec)$ be a finite poset, and let \ts $z_1 \prec z_2 \prec z_3$ \ts
be three distinct elements in~$X$.
Suppose that \. $\aFr(k,\ell+2)=\aFr(k+2,\ell)=0$ \. and \. $\aFr(k,\ell) \. \aFr(k+1,\ell+1)> 0$.
Then
		\[ \aFr(k,\ell+1) \. \aFr(k+1,\ell) \ = \  \aFr(k+1,\ell+1) \. \aFr(k,\ell).  \]
\end{lemma}

\begin{proof}
As in the proof of Corollary~\ref{c:vanish-either}, let \. $\cS:= \big\{(k,\ell)\in \nn^2 \. : \. \aF(k,\ell)>0\big\}$ \. denote the support
of \. $\aF(\cdot,\cdot)$.   By the assumption, we have \. $(k,\ell+2), \ts (k+2,\ell)\not \in \cS$ \. and \.
$(k,\ell), (k+1,\ell+1)\in \cS$.  Theorem~\ref{thm:vanishing} then gives:
$$k +1  \. \leq \.   n+1-b(z_1)-b^*(z_2) \quad \text{and} \quad k +2  \. > \.   n+1-b(z_1)-b^*(z_2),
$$
$$\ell+1  \. \leq \.  n+1 -b(z_2) -b^*(z_3) \quad \text{and} \quad \ell + 2 \. > \.   n+1 -b(z_2) -b^*(z_3).
$$
Together these imply
$$(\ast) \qquad k \. = \.  n - b(z_1) - b^*(z_2) \quad \text{and} \quad \ell \. = \. n-b(z_2) -b^*(z_3). \qquad \quad \ $$
Theorem~\ref{thm:vanishing} also gives
$$k+\ell +2 \, \leq \, n+1 -b^*(z_3) - b(z_1).$$
Substituting \ts $(\ast)$ \ts into this inequality, we get:
$$n - b(z_1) - b^*(z_2) + n-b(z_2) -b^*(z_3) \, \leq \, n- 1 -b^*(z_3) - b(z_1).
$$
This simplifies to \. $n+1 \leq b(z_2) +b^*(z_2)$ \. and  implies that all elements in $X$ are comparable to~$z_2$.
	
Let \. $S= B(z_2)-z_2$ \. and \. $T =B^*(z_2)-z_2$ \. be the lower set and upper sets of $z_2$, respectively.
Denote \. $s:=|S| = b(z_2)-1$ \. and \. $t:=|T|=b^*(z_2)-1$\ts.
Note that \. $X=S \sqcup T \sqcup \{z_2\}$ \. by the argument above.

Let \. $1\le r \le n$. Consider a subposet \. $(S,\prec)$ \. of \. $P=(X,\prec)$ \. and
denote by \. $\aN_{r}$ \. the number of linear extensions \ts $L$ \ts of
\. $(S,\prec)$ \. such that \. $L(z_1)=r$.    Similarly, consider a subposet \. $(T,\prec)$ \. of \. $P=(X,\prec)$ \. and
denote by \. $\aN_{r}'$ \. the number of linear extensions \ts $L$ \ts of
\. $(S,\prec)$ \. such that \. $L(z_3)=r$.

Since \. $z_1 \prec z_2\prec z_3$, we have \. $z_1 \in S$ \. and \. $z_3 \in T$.  Therefore, for all \.
$p,q\ge 1$ \. we have:
	\[  \aF(p,q) \ = \  \aN_{s-p+1} \, \aN_{q}'\.. \]
	This implies that
	\begin{align*}
		& \aF(k,\ell+1) \. \aF(k+1,\ell) \ = \ \aN_{s-k+1} \, \aN_{\ell+1}' \, \aN_{s-k} \, \aN_{\ell}' \\
& \qquad = \ \aN_{s-k} \, \aN_{\ell+1}' \, \aN_{s-k+1} \, \aN_{\ell}' \ = \
		\aF(k+1,\ell+1) \. \aF(k,\ell),
	\end{align*}
	as desired.
\end{proof}

\smallskip

\begin{ex} \label{ex:vanish-unequal}
For \. $k, \ell\ge 1$, let \. $X:=\{x_1,\ldots,x_{k+\ell-1},z_1,z_2,z_3\}$.  Consider a poset \.
$P=(X,\prec)$, where \. $A:=\{x_1,\ldots,x_{k+\ell-1},z_2\}$ \. is an antichain,
and \. $z_1\prec A \prec z_3$.  Observe that
$$
\aF(k,\ell)= \aF(k+1,\ell+1)=\aF(k,\ell+2)=\aF(k+2,\ell)=0,$$
$$
\aF(k,\ell+1) \. = \. \tbinom{k+\ell-1}{k-1} \quad \text{and} \quad
 \aF(k+1,\ell) \. = \. \tbinom{k+\ell-1}{k}\ts.
$$
Then we have:
		\[ \aF(k,\ell+1) \. \aF(k+1,\ell) \, = \, \tbinom{k+\ell-1}{k-1}\tbinom{k+\ell-1}{k} \, > \,  \aF(k+1,\ell+1) \. \aF(k,\ell) \, = \, 0.  \]
This shows that the nonvanishing assumption \.
$\aF(k,\ell) \. \aF(k+1,\ell+1)> 0$ \. in Lemma~\ref{lem:strict-CPC-zero}  cannot be dropped.
\end{ex}

\bigskip

\section{Cross product inequalities in the nonvanishing case } \label{s:words}

\subsection{Algebraic setup} \label{ss:words-setup}
We employ the algebraic framework from \cite[$\S$6]{CPP2}.
With every linear extension \ts $L\in \Ec(P)$ \ts we associate a  word \. $\xx_L \ts=\ts x_1\ldots x_n\ts \in X^\ast$,
such that \. $L(x_i) =i$ \. for all \. $1\le i \le n$.  In the notation of the previous section, this says
that \. $X=\{x_1,\ldots,x_n\}$ \. is a \emph{natural labeling} \ts corresponding to~$L$.
%

We can now define the following action of the group \ts $\RG_n$ \ts on \ts $\Ec(P)$ \ts as the right
action on the words \ts $\xx_L$, \ts $L\in \Ec(P)$.  For \. $\xx_L \ts =\ts x_1\ldots \ts x_n$ \. as above, let
\begin{equation}\label{eq:tau-def}
(x_1\ldots \ts x_n) \. \tau_i \ := \ \begin{cases} \ x_1 \ldots \ts x_n, & \ \text{if \, $x_i \prec x_{i+1}$}\ts,\\
\ x_1\dots x_{i+1} \ts x_i \dots x_n\ts, & \ \text{if \, $x_i \parallel x_{i+1}$}\ts.
\end{cases}
\end{equation}

\medskip

\subsection{Single element ratio bounds} \label{ss:words-single}
Let \ts $P=(X,\prec)$ \ts  be a poset with \ts $|X|=n$ elements,
and fix an element \ts $a\in X$ \ts of the poset.
Let \ts $\Nc_k$ \ts be the set of linear extensions \ts $L\in \Ec(P)$ \ts
such that \ts $L(a)=k$, and let \ts $\aN_k :=|\Nc_k|$.

\smallskip

\begin{lemma}\label{lem:Sta-up1}
We have:
\begin{alignat*}{2}
		 \frac{\aNr_{k}}{\aNr_{k-1}} \  &\leq \ t(a) \quad &&\text{if} \quad \aNr_{k-1}>0, \quad \text{and}\\
		  \frac{\aNr_{k}}{\aNr_{k+1}} \  &\leq \  t^*(a)  \quad &&\text{if} \quad \aNr_{k+1}>0.
\end{alignat*}
\end{lemma}

The idea and basic setup of the proof will be used throughout.

\begin{proof}
Consider the first inequality.
The main idea is to construct an explicit injection \. $\phi: \Nc_k \to \Nc_{k-1} \times I$, where $I:= \{1,\ldots,t(a)\}$. This will show that $\aN_k = | \Nc_k | \leq | \Nc_{k-1} \times I | =\aN_{k-1} t(a)$.

We identify a linear extension $L$ where $L(a)=k$ with a word \. $\xx \in \Nc_k$ \. where \. $x_k =a$.  Let $x_i$ be the last element in $\xx$ appearing before $a$ which is incomparable to $a$, that is set $i:=\max \ts \{\ts i \. : \. i <k, \. x_i \not\prec x_k \ts\}$. Such element exists because $\aN_{k-1} >0$ implies that \ts $b(a)\leq k-1$ and so among $x_1,\ldots,x_{k-1}$ there is at least one $x_i \not \prec a$. Moreover,
since $i$ is maximal, we must have $x_j \prec  x_k$ \ts for $j \in [i+1,k]$. Also, for $j\in [i+1,k]$ we must have $x_j \ts\|\ts{} x_i$, as otherwise we would have $x_i \prec x_j \prec x_k=a$.
 Thus, we have
\. $x_j \in \cL(a,x_i)$ \. for \. $i <j<k$ and so  \ts $1\leq k-i \leq t(a)$.

We now define \. $\phi(\xx) := (\xx \tau_i \cdots \tau_{k-1} , k-i)$. Since $x_i \ts\|\ts{} x_j$ for $j\in [i+1,\ldots,k]$ we have that $x_i$ is transposed consecutively with $x_{i+1},\ldots,x_k$, so $\xx \tau_i \cdots \tau_{k-1} = x_1\ldots x_{i-1}x_{i+1}\ldots x_k x_i x_{k+1}\ldots \in \Nc_{k-1}$. We record the original position of $x_i$ via $k-i$.

To see this is an injection we construct \ts $\phi^{-1}$, if it exists. Namely, $\phi^{-1}(\xx',r)$
moves the element $x'_k$ after \ts $x'_{k-1}=a$ \ts forward by \ts $r=(k-i)$ \ts positions as long as $x'_k \ts\|\ts{}  x_j$ for $j \in [k-r,k-1]$.
This completes the proof of the first inequality.
The second inequality follows by applying the same argument to the dual poset~$P^*$.
\end{proof}

\smallskip

\begin{cor}\label{cor:Sta-upb}
We have:
\begin{alignat*}{2}
	\frac{\aNr_{k}}{\aNr_{k-1}} \  &\leq \ k-1 \quad &&\text{if} \quad \aNr_{k-1}>0, \quad \text{and}\\
	\frac{\aNr_{k}}{\aNr_{k+1}} \  &\leq \  n-k  \quad &&\text{if} \quad \aNr_{k+1}>0.
\end{alignat*}
\end{cor}
\smallskip

Note that the inequalities in the corollary are tight, see Proposition~\ref{prop:Stanley-up}.

\smallskip

\begin{proof}
	Observe that  \. $t(a) \leq k-1$ \. since there are  at most \. $(k-1)$ \. elements less than or equal to~$a$ by the assumption that $\aN_{k-1}>0$.
	Similarly, observe that  \. $t^*(a) \leq n-k$ \. since there are  at most \. $(n-k)$ \.
elements greater than or equal to~$a$ by the assumption that $\aN_{k+1}>0$.  These imply the result.
\end{proof}

\medskip

\subsection{Double element ratio bounds} \label{ss:words-double}
We now give bounds for nonzero ratios of \ts $\aF(k,\ell)$.
For the degenerate case, see Section~\ref{s:vanish}.

\smallskip

\begin{lemma}\label{l:Fkell-bound1}
Suppose that \. $\aFr(k,\ell+2)>0$. Then we have:
\begin{align*}
\frac{ \aFr(k+1,\ell+1) }{\aFr(k,\ell+2)} \, \leq \, \min\{ t(z_2), k \} \, + \, \min\big\{ b(z_1,z_2)-2, t^*(z_1)\big\} \. \cdot \.
\big( t^*(z_3) +  t(z_2)\big).
\end{align*}
Similarly, suppose that \. $\aFr(k+2,\ell)>0$. Then we have:
\begin{align*}
\frac{ \aFr(k+1,\ell+1) }{\aFr(k+2,\ell)} \, \leq \, \min\big\{ t^*(z_2), \ell \big\} \, + \, \min\big\{ b(z_2,z_3)-2,t(z_3)\big\} \. \cdot \.
\big( t(z_1) +  t^*(z_2)\big).
\end{align*}
\end{lemma}

\begin{proof}
For the first inequality, we construct an injection \. $\psi: \cF(k+1,\ell+1) \to I \times \cF(k+2,\ell)$, where \. $I=I_1 \sqcup I_2 \sqcup I_3$ \. and \ts $I_i$ \ts are intervals of lengths given by the RHS (see below).  We use notation \. $[p,q]=\{i\in \nn \. : \. p \le i \le q\}$ \. to denote the integer interval.

Let \. $\xx \in \cF(k+1,\ell+1)$ \. be a word, such that \. $x_i=z_1$, \. $x_{i+k+1} =z_2$ \. and \.  $x_{i+k+\ell+2} =z_3$. We consider several cases.

\smallskip

\nin
{\bf Case 1:} \ts Suppose there exists an element \. $x_j \not \prec z_2$ \. for some \. $j \in [i+1,i+k]$. Let $j$ be the maximal such index. Then for every \ts $r \in [j+1,i+k]$ \ts we have that $x_r \in \cL(z_2,x_j)$. Set $\psi(\xx) = (\xx \tau_j \cdots \tau_{i+k}, i+k+1-j)$, i.e.\ \ts $\psi$ \ts moves \ts $x_j$ \ts to the position after \ts $z_2$, so that \ts $z_2$ \ts is now in position~$\ts i+k$. Observe that the inverse of~$\ts \psi$ \ts  exists for all \. $\yy \in \cF(k,\ell+2)$,  since \.  $y_{i+k}=z_2\ts\|\ts{}y_{i+k+1}$. Note that \. $i+k+1-j \leq  \min\{ u(z_2,x_j),  k \}$.  Thus, we can record the value \. $(i+k+1-j)$ \. in the first interval \. $I_1 =[1, \min\{t(z_2),k\}]$.

\smallskip

\nin
{\bf Case 2:} \ts Suppose that we have \. $x_j \prec z_2$ \. for all \. $j\in [i,i+k]$.  Then there exists an element \. $x_j \not \succ z_1$\ts. Indeed, otherwise \. $x_j \in B(z_1,z_2)$ \. for all \. $j \in [i,i+k+1]$, which gives \. $k+2 \leq |B(z_1,z_2)|$ \. and implies \. $\aF(k,\ell+2)=0$ \. contradicting the assumption.  As above, let $j$ be the smallest possible index such that $x_j\ts\|\ts{}z_1$, so we can move $x_j$ in front of $z_1$. Note that $j-i \leq  \min\{ b(z_1,z_2)-2, u^*(z_1,x_j)\}$. We now have a word $\xx' \in \cF(k,\ell+1)$.
We split this case into two subcases.

\smallskip

\nin
{\bf Subcase 2.1:} \ts  Suppose there exists \. $x_r \not \succ z_3$ \. for \. $r > i+k+\ell+2$. Let $r$ be the minimal such index, and move $x_r$ in front of $z_3$, creating a word \. $\xx'' \in \cF(k,\ell+2)$. Note that \. $r-(k+\ell+2+i) \leq  u^*(z_3,x_r)$.  Thus, we can record the value \. $(j-i,r-(k+\ell+2+i))$ \. in the second interval \. $I_2=[1,\min\{ b(z_1,z_2)-2, t^*(z_1)\}t^*(z_3)]$.

\smallskip

\nin
{\bf Subcase 2.2:} \ts Suppose \. $x_s \succ z_3$ \. for all \. $s > i+k+\ell+2$. Then, since \. $\aF(k,\ell+2) \neq 0$, there must be some $x_s \not \prec z_2$, for $s< i+k+1$.  Since we are in Case~2, we have \. $s< i$. Let $s$ be the largest such index. Thus \. $x_{s+1},\ldots,x_{i+k} \prec x_{i+k+1}= z_2$. We can then move \ts $x_s$ \ts past all these entries to right past \ts $z_2$ \ts and obtain a word in \. $\cF(k,\ell+2)$. Note that \. $i-s \leq u(z_2,x_s)$. Thus, we can record the value \.
$(j-i, i-s)$  \. in the third interval \. $I_3=[\min\{ b(z_1,z_2)-2, t^*(z_1)\} t(z_2)]$.

\smallskip

Gathering these cases, and noting that \. $t(x) \geq u(x,y)$ \. and \. $t^*(x) \geq u^*(x,y)$ \. for all $x,y\in X$,
we obtain the desired first inequality.  For the second inequality, we apply the analogous argument to
the dual poset~$P^*$.
\end{proof}

\medskip

\subsection{Bounds on cross product ratios} \label{ss:words-cross}
We can now bound the cross product ratios in the nonvanishing case.

\smallskip

\begin{cor}\label{cor:thin}
Let \ts $P=(X,\prec)$ \ts be either a $t$-thin or $t$-flat poset with respect to \. $\{z_1,z_2,z_3\}$.
Suppose that \. $\aFr(k,\ell+2)>0$. Then we have:
$$\aFr(k+1,\ell+1) \, \leq \, \aFr(k,\ell+2) \. \cdot \. \min\big\{k \ts (2t+1), 2t^2+t\big\}.$$
Similarly, suppose that \. $\aFr(k+2,\ell)>0$.  Then we have:
$$\aFr(k+1,\ell+1)  \, \leq \,  \aFr(k+2,\ell)   \. \cdot \. \min\big\{\ell \ts (2t+1), 2t^2+t\big\}.$$
\end{cor}

\begin{proof}
These inequalities come from different choices in the
minima on the RHS of inequalities in Lemma~\ref{l:Fkell-bound1}.
\end{proof}

\smallskip

\begin{thm}\label{t:thin-part1}
Let \ts $P=(X,\prec)$ \ts be either a $t$-thin or $t$-flat poset with respect to \. $\{z_1,z_2,z_3\}$.
Suppose also that \. $\aFr(k,\ell+2) \. \aFr(k+2,\ell) >0$. Then:
\begin{align*}
\frac{\aFr(k+1,\ell)\aFr(k,\ell+1)}{\aFr(k,\ell)\aFr(k+1,\ell+1)} \ \geq \
\max \left\{\frac{1}{2} \, + \, \frac{1}{2\ts \sqrt{k\ell} \ts (2t+1)} \ , \,
\frac{1}{2} \, + \,\frac{1}{2 \ts (2t^2+t)}\right\}.
\end{align*}
\end{thm}

\smallskip

\begin{proof}
	These inequalities follow from Proposition~\ref{prop:cpc1+eps} and the inequalities in Corollary~\ref{cor:thin}.
\end{proof}

\smallskip

\begin{thm}\label{t:main-part1}
Suppose that \. $\aFr(k,\ell+2) \. \aFr(k+2,\ell) > 0$. Then:
$${\aFr(k+1,\ell)\. \aFr(k,\ell+1)} \, \geq \, {\aFr(k,\ell) \. \aFr(k+1,\ell+1)} \bigg( \frac12 \, + \, \frac{1}{2 \ts \sqrt{(2nk-2n-k+2)(2n\ell-2n-\ell+2)} }\bigg).$$
\end{thm}

\smallskip

\begin{proof}
	It follows from the definition that \. $t(x), t^*(x) \leq n-1$ \. for every \ts $x \in X$.
	The nonvanishing condition in the assumption, combined with Theorem~\ref{thm:vanishing}
implies that \. $b(z_1,z_2)\leq k+1$ \. and \. $b(z_2,z_3) \leq \ell+1$.  It then follows from Lemma~\ref{l:Fkell-bound1} that
	\[ \frac{ \aF(k+1,\ell+1) }{\aF(k,\ell+2)}  \ \leq \  k+(k-1) (2t) \ \leq \ k+(k-1)(2n-2) \ = \ 2nk-2n-k+2.\]
	Similarly, we have:
		\[ \frac{ \aF(k+1,\ell+1) }{\aF(k+2,\ell)}  \ \leq \     2n\ell-2n+1.\]
	The theorem now follows from Proposition~\ref{prop:cpc1+eps}.
\end{proof}

\bigskip

\section{Cross product inequalities in the vanishing case } \label{s:cross-vanish}

\subsection{Double element ratio bounds} \label{ss:vanish-double}
As before, let \. $P=(X,\prec)$ \. be a poset with \. $|X|=n$ \. elements, and let
\. $z_1\prec z_2 \prec z_3$ \. be distinct elements in~$X$.
The following are the counterparts of the
cross product inequalities in~$\S$\ref{ss:words-double}.

\smallskip

\begin{lemma}\label{l:Fkell-bound2}
Suppose that \. $\aFr(k,\ell)> 0$. Then:
$$\aligned
& \frac{\aFr(k+1,\ell)}{\aFr(k,\ell)} \ \leq \ \min\{ k, t^*(z_1)\} \,  +\, \min\{k, t(z_3)-1\}\,  +\, \\
& \qquad + \, \min\{b(z_1,z_2)-1,t(z_2)\} \. \big(\min\{\ell-1,t(z_3)\} + \min\{\ell-1,t^*(z_1)-1\}\big).
\endaligned$$
\end{lemma}

\smallskip

Note that the nonvanishing condition implies that \. $b(z_1,z_2)\leq k+1$ \. and \. $b(z_2,z_3) \leq \ell+1$.

\smallskip

\begin{proof}
	We proceed as in the proof of Lemma~\ref{l:Fkell-bound1}, constructing an injection
\. $\psi: \cF(k+1,\ell) \to I \times \cF(k,\ell)$, where \. $I=I_1\sqcup I_2 \sqcup I_{3} \sqcup I_4$ \.
are intervals of lengths specified by the RHS, each of them given in the corresponding case below.

Let $\xx \in \cF(k+1,\ell)$ be a word (corresponding to a linear extension) with \. $x_i=z_1$, \.
$x_{i+k+1}=z_2$ \. and \. $x_{i+k+\ell+1}=z_3$.
We consider several independent cases, which correspond to different parts of the interval~$I$:

\smallskip

\nin
{\bf Case 1:} \ts Suppose that there exists \. $x_j \ts\|\ts{} z_1$ \. with \. $j \in [i+1,i+k]$ \. and let $j$ be the minimal such index. Then \. $\{x_i,\ldots,x_{j-1}\} \subseteq \cL^*(z_1,x_j)$ \. and  \. $j-i \leq \min\{k,t^*(z_1)\}$. Take \. $\xx \tau_{j-1} \cdots \tau_i$, which moves $x_j$ to position $i$ and $z_1$ to position $i+1$.  Then the resulting word is in $\cF(k,\ell)$, and we record the value \. $(j-i)$ \. in the first interval $I_1=[1,\min\{ k, t^*(z_1)\}]$.

\smallskip

\nin
{\bf Case 2:} \ts Suppose that \. $x_j \succ z_1$ \. for all \. $j \in [i+1,i+k]$.  Furthermore, suppose that  \.
$x_r \succ z_1$ \. and \. $x_r \prec z_3$ \. for all \. $r \in [i+k+2,i+k+\ell]$.
These assumptions imply that there exists \. $j \in [i+1,i+k]$ \. such that \. $x_j\ts\|\ts{} z_3$, as otherwise we have
\. $\{x_{i}, \ldots, x_{i+k+\ell+1}\} \in B(z_1,z_3)$, contradicting the assumption that \. $\aF(k,\ell)>0$.
Assume that  $j$ is the maximal such index~$j$.  It then follows that \. $\{x_{j+1},\ldots, x_{i+k+\ell+1}\} \subseteq \cL(z_3,x_j) $.
This implies that \. $i+k+\ell+1-j \leq t(z_3)$, which in turn implies that \.
 $i+k+1-j\leq t(z_3)-\ell \leq t(z_3)-1$.
Then we take \. $\xx' = \xx \tau_{j}\cdots \tau_{i+k+\ell+1} \in\cF(k,\ell)$ \. and record the value \.
$(i+k+1-j)$ \. in the second interval \. $I_2=[1,\min\{k, t(z_3)-1\}]$.

\smallskip

\nin
{\bf Case 3:} \ts Suppose again that \.$x_j \succ z_1$ \. for all \. $j \in [i+1,i+k]$, but now  that there exists \.
$r \in [i+k+2,i+k+\ell]$ \. such that either \. $x_r \ts\|\ts{} z_1$ \. or \. $x_r \ts\|\ts{} z_3$.
The first condition implies that there exists \. $x_j \ts\|\ts{}z_2$ \. with \. $j\in [i+1,i+k]$, as otherwise we would have \. $\aF(k,\ell)=0$.
 Let $j$ be the maximal such index.  Then \.
 $\{x_{j+1},\ldots,x_{i+k+1} \} \subseteq B(z_1,z_2)-z_1$, and thus \. $i+k+1-j \leq b(z_1,z_2)-1$.
 Also note that \.
 $\{x_{j+1},\ldots,x_{i+k+1} \} \subseteq U(z_2,x_j)$, and thus \. $i+k+1-j \leq t(z_2)$.
 Move \ts $x_j$ \ts right past \ts $z_2$ \ts via \. $\xx \tau_j \cdots \tau_{i+k+1}$ \. and record that move with \.
 $s:=i+k+1-j \leq \min\{b(z_1,z_2)-1,t(z_2)\}$.  We now consider the new word \. $\xx' \in \cF(k,\ell+1)$.
 We split this case into two subcases.

\smallskip

\nin
{\bf Subcase 3.1:} \ts Suppose that there exists an element \. $x_r'=x_r \ts\|\ts{}z_3$ \. for some \. $r \in [i+k+2,i+k+\ell]$.
Let $r$ be the maximal such index.  Then \. $\{ x'_{r+1},\ldots,x'_{i+k+\ell+1} \} \subseteq \cL(z_3,x'_r)$ \. and \. $i+k+\ell+1-r \leq t(z_3)$.
 We then create the word \. $\xx' \tau_r \cdots \tau_{i+k+\ell+1}\in \cF(k,\ell)$ \. where \ts $x'_r$ \ts is moved past \ts $z_3$. We record the pair \. $(s,i+k+\ell+1-r)$ \. in the product of intervals \. $I_3=[1,\min\{b(z_1,z_2)-1,t(z_2)\}] \times [1,\min\{\ell-1,t(z_3)\}]$.

\smallskip

\nin
{\bf Subcase 3.2:} \ts Suppose that there exists \.
$x_r'=x_r \ts\|\ts{}z_1$ \. for \. $r \in [i+k+2,i+k+\ell]$.  We take the minimal such \ts $r$.
Then \. $\{x'_{i},\ldots, x'_{r-1}\} \subseteq \cL^*(z_1, x_r')$ and thus $r-i\leq t^*(z_1)$.
This in turn implies that \. $r-i-k-1\leq t^*(z_1)-k-1 \leq t^*(z_1)-1$.
Take a word \. $\xx''\in \cF(k,\ell)$ \. by moving \ts $x_r'$ \ts to the position
before \ts $z_1$ \ts and record the pair \. $(s,r-i-k-1)$ \. in the product of intervals
\. $I_4=[1,\min\{b(z_1,z_2)-1,t(z_2)\}] \times  [1,\min\{\ell-1,t^*(z_1)-1\}]$.

\smallskip

Gathering these cases we obtain the desired inequality in the lemma.\end{proof}

\smallskip

\begin{lemma}\label{l:Fkl-bound3}
Suppose that \. $\aFr(k+2,\ell)> 0$. Then:
$$\frac{\aFr(k+1,\ell)}{\aFr(k+2,\ell)} \ \leq \ t(z_1) \. + \. \big(t^*(z_2)-1\big) \. + \. \min\big\{\ell-1,t^*(z_2)\big\} \. t^*(z_3). $$
\end{lemma}
\smallskip

\begin{proof}
	We proceed as in the proof of Lemma~\ref{l:Fkell-bound1}, constructing an injection \.
$\psi: \cF(k+1,\ell) \to I \times \cF(k+2,\ell)$, where \.
$I=I_1\sqcup I_2 \sqcup I_{3}$ \. are intervals of lengths specified
by the RHS corresponding to each case below.
	
Let \. $\xx \in \cF(k+1,\ell)$ \. be a word (corresponding to a linear extension) with \.
$x_i=z_1$, \. $x_{i+k+1}=z_2$ \. and \. $x_{i+k+\ell+1}=z_3$.
We consider three independent cases, which correspond to different intervals~$I_i$ (see below).

\smallskip

\nin
{\bf Case 1:} \ts Suppose that there exists \. $x_j \ts\|\ts{} z_1$ \. with \. $j \in [1,i-1]$, and let $j$ be the maximal such index. Then $\{x_{j+1},\ldots,x_{i}\} \subset \cL(z_1,x_j)$ and so  $i-j \leq t(z_1)$. We take \. $\xx \tau_{j} \cdots \tau_{i-1}$, which moves
\ts $x_j$ \ts to position \ts $i$ \ts and \ts $z_1$ \ts to position \ts $i-1$. Then the resulting word is in \ts
$\cF(k+2,\ell)$, and we record the value \ts $(i-j)$ \ts in the first interval \. $I_1=[1,t(z_1)]$.
	
\smallskip

\nin
{\bf Case 2:} \ts Suppose that	\. $x_j \prec z_1$ \. for all \. $j \in [1,i-1]$.
	Since \. $\aF(k+2,\ell)>0$, there exists \. $x_j \ts\|\ts{} z_2$ \. with \. $j \in [ i+k+2,n]$.
	Let $j$ be the minimal such index. Then \.
$\{x_{i+k+1},\ldots,x_{j-1} \} \subset \cL^*(z_2,x_j)$, and thus \.  $j-i-k-1 \leq t^*(z_2)$.
Move \ts $x_j$ \ts to the front of \ts $z_2$ \ts via \ts $\xx \tau_{j-1} \cdots \tau_{i+k+1}$ \ts
	to get a new word~$\ts\xx'$. We split this case into two subcases:
	
\smallskip

\nin
{\bf Subcase 2.1:} \ts Suppose that	\.  $j \in [i+k+\ell+2,n]$.  Then \. $\xx' \in \cF(k+2,\ell)$.
	Also note that \. $j-i-k-\ell-1 \leq t^*(z_2)-\ell \leq t^*(z_2)-1$.
	We then record the value \. $(j-i-k-\ell-1)$ \. in the second interval \. $I_2=[1,t^*(z_2)-1]$.
	
\smallskip

\nin
{\bf Subcase 2.2:} \ts Suppose that	\.  $j \in [i+k+2,i+k+\ell]$.  Then \. $\xx' \in \cF(k+2,\ell-1)$.
	By the assumption of Case~2 and the fact that \. $\aF(k+2,\ell)>0$,
	there exists \. $r \in [i+k+\ell+2,n]$ \. such that \. $x_r' \ts\|\ts{} z_3$.
	Assume that \ts $r$ \ts is the minimal such index.  It then follows that \.
	$\{x'_{i+k+\ell+1},\ldots, x'_{r-1}\} \subseteq \cL^*(z_3,x_r')$.  This implies that \.
$(r-i-k-\ell-1) \leq t^*(z_3)$.
	Move \ts $x_r'$ \ts to the front of \ts $z_3$ \ts to obtain a new word \. $\xx'' \in \cF(k+2,\ell)$,
	and we record the value \ts $(j-i-k-1,r-i-k-\ell-1)$ \ts to the product of intervals
\ts $I_3 = [1,\min\{\ell-1,t^*(z_2)\}] \times [1,t^*(z_3)] $.

\smallskip

Gathering these cases we obtain the desired inequality in the lemma.\end{proof}

\medskip

\subsection{Bounds on cross product ratios} \label{ss:vanish-cross}
%
%
%
We are now ready to obtain bounds on the cross product ratios in the vanishing
case.

\smallskip

\begin{prop}\label{p:Fk-bound1}
Suppose that \. $\aFr(k,\ell) \. \aFr(k+2,\ell)> 0$.
Then
$$\frac{ \aFr(k,\ell) \. \aFr(k+2,\ell)}{\aFr(k+1,\ell)^2} \ \geq \ \frac{1}{2n\ell^2 k}\..
$$
\end{prop}

\begin{proof}  First, observe that \. $b(z_1,z_2)\leq k+1$ \. and \. $t(z_1) + t^*(z_2) \leq b(z_1)+b^*(z_2) \leq n$.
We then have:
\begin{equation}\label{eq:Fk-bound1}
\aligned
 & \min\{b(z_1,z_2)-1,t(z_2)\}\ts \big(\min\{\ell-1,t(z_3)\} \. + \. \min\{\ell-1,t^*(z_1)-1\}\big)\\
 & \qquad + \, \min\{ k, t^*(z_1)\} \, +\, \min\{k, t(z_3)-1\} \ \leq \   k (2\ell-2) + 2k \ = \ 2k\ell
 \endaligned
\end{equation}
and
\begin{equation}\label{eq:Fk-bound2}
t(z_1) \, + \, (t^*(z_2)-1) \,  + \, \min\big\{\ell-1,t^*(z_2)\big\} \. t^*(z_3) \ \leq \ n-1 \. + \. (\ell-1)(n-1) \ < \ n\ell \ts.
\end{equation}
Lemmas~\ref{l:Fkell-bound2} and~\ref{l:Fkl-bound3} now give:
$$\frac{ \aF(k,\ell) \. \aF(k+2,\ell)}{\aF(k+1,\ell)^2} \ \geq \ \left(\frac{1}{n\ell}\right) \cdot \left(\frac{1}{2k\ell}\right),
$$
as desired.
\end{proof}

\smallskip

We also need the following variation on this proposition.

\smallskip

\begin{prop}\label{p:Fk-bound2}
Let \ts $P=(X,\prec)$ \ts be either a $t$-thin or $t$-flat poset with respect to \. $\{z_1,z_2,z_3\}$.
Suppose also that \. $\aFr(k,\ell) \. \aFr(k+2,\ell) >  0$. Then we have:
	\begin{align*}
		\frac{\aFr(k,\ell)\aFr(k+2,\ell)}{ \aFr(k+1,\ell)^2 }
        \ \geq \ \max \left\{\ts\frac{1}{2k\ell(\ell+1)t} \ , \, \frac{1}{2t(t+1)^3} \ts \right\}.
	\end{align*}
\end{prop}

\begin{proof}
We follow the proof of the proposition above with the following adjustments.
For the first inequality in the maximum,
we replace the bound \eqref{eq:Fk-bound2} with the following:
\begin{equation}\label{eq:Fk-bound3}
t(z_1) \. + \. (t^*(z_2)-1) \.  + \. \min\big\{\ell-1,t^*(z_2)\big\} \. t^*(z_3) \,
\leq \, 2t-1 \. + \. (\ell-1)(t-1) \, < \, (\ell+1) t \ts.
\end{equation}
Now the first inequality follows from Lemmas~\ref{l:Fkell-bound2} and~\ref{l:Fkl-bound3},
with the parameters bounded by
\eqref{eq:Fk-bound1} and \eqref{eq:Fk-bound3}.

For the second inequality in the maximum, we replace the bound
\eqref{eq:Fk-bound1} and \eqref{eq:Fk-bound2} with the following:
\begin{equation}\label{eq:Fk-bound4}
\aligned
 & \min\{b(z_1,z_2)-1,t(z_2)\}\ts \big(\min\{\ell-1,t(z_3)\} \. + \. \min\{\ell-1,t^*(z_1)-1\}\big)\\
 & \qquad + \, \min\{ k, t^*(z_1)\} \, +\, \min\{k, t(z_3)-1\} \ \leq \   t(2t-1) + t + (t-1)  \, < \, 2t(t+1)
 \endaligned
\end{equation}
and
\begin{equation}\label{eq:Fk-bound5}
t(z_1) \, + \, (t^*(z_2)-1) \,  + \, \min\big\{\ell-1,t^*(z_2)\big\} \. t^*(z_3) \ \leq \ 2t-1 + t^2 \, < \, (t+1)^2.
\end{equation}
Now the second inequality follows from Lemmas~\ref{l:Fkell-bound2} and~\ref{l:Fkl-bound3},
with the parameters bounded by
\eqref{eq:Fk-bound4} and \eqref{eq:Fk-bound5}.
\end{proof}

\smallskip

\begin{thm}\label{t:main-part2}
Suppose that \. $\aFr(k+2,\ell)> 0$ \. and \. $\aFr(k,\ell+2)=0$.  Then we have:
$$
{\aFr(k+1,\ell)\. \aFr(k,\ell+1)} \ \geq \ {\aFr(k,\ell)\. \aFr(k+1,\ell+1)} \left( \frac12 \. + \. \frac{1}{16nk\ell^2} \right).
$$
\end{thm}

\begin{proof}
	We can assume that \. ${\aF(k,\ell)\. \aF(k+1,\ell+1)}>0$ \.
	as otherwise the result is trivial.
	Propositions~\ref{p:cpc+eps0} and~\ref{p:Fk-bound1} then give:
\begin{align*}
\frac{ \aF(k+1,\ell)\aF(k,\ell+1) }{\aF(k+1,\ell+1)\aF(k,\ell)}  \ \geq \ \bigg(1 + \sqrt{1 - \frac{1}{2nk\ell^2} } \.\bigg)^{-1}
 \ \geq \ \frac12 \, + \, \frac{1}{16nk\ell^2}\,,
\end{align*}
where the last inequality follows from \. $\frac1{1+\sqrt{1-\al}} \ge \frac{1}{2}+ \frac{\al}{8}$ \. for \ts $0\le \al< 1$.
\end{proof}

\smallskip

\begin{thm}\label{t:thin-part2}
Let \ts $P=(X,\prec)$ \ts be either a $t$-thin or $t$-flat poset with respect to \. $\{z_1,z_2,z_3\}$.
	Suppose also that \. $\aFr(k,\ell+2)=0$ \. and \. $\aFr(k+2,\ell) >  0$. Then we have:
$$
{\aFr(k+1,\ell) \. \aFr(k,\ell+1)} \ \geq \ {\aFr(k,\ell) \. \aFr(k+1,\ell+1)} \. \max
\left\{\frac12 \. + \. \frac{1}{16k\ell (\ell+1)t} \ ,\,\frac12 \. + \. \frac{1}{16t(t+1)^3} \right\}.$$
\end{thm}

\smallskip

\begin{proof}
	The proof follows the same argument as in Theorem~\ref{t:main-part2}, where Proposition~\ref{p:Fk-bound2} is used in place of Proposition~\ref{p:Fk-bound1}.
\end{proof}

\medskip

\subsection{Putting everything together} \label{ss:vanish-proofs}
We can now combine the results to finish the proofs.

\smallskip

\begin{proof}[Proof of Main Theorem~\ref{t:main}]
The first inequality~\eqref{eq:main-thm-1} follows immediately from
Theorem~\ref{t:main-part1}.  The second inequality~\eqref{eq:main-thm-2} follows
immediately from Theorem~\ref{t:main-part2}.  The third inequality
\eqref{eq:main-thm-2-swap} follows by the symmetry \. $P \lra P^\ast$,
\. $z_1 \lra z_3$ \. and \.  $k\lra \ell$. Finally, the
equality~\eqref{eq:main-thm-3} is the equality in
Lemma~\ref{lem:strict-CPC-zero}.
\end{proof}

\smallskip

\begin{proof}[Proof of Theorem~\ref{t:thin}]
The proof of~\eqref{eq:thin-thm} follows the previous proof.
The result is trivial in the case \. $\aF(k,\ell) \. \aF(k+1,\ell+1) = 0$.
In the vanishing case \. $\aF(k,\ell+2) = \aF(k+2,\ell)=0$ \. and \.
$\aF(k,\ell) \. \aF(k+1,\ell+1) > 0$ \. the result follows from the equality in
Lemma~\ref{lem:strict-CPC-zero}.  In the case when only one of the terms
is vanishing: \. $\aF(k,\ell+2)=0$ \. and \. $\aF(k+2,\ell) > 0$, the result
is given by Theorem~\ref{t:thin-part2}. The case \. $\aF(k+2,\ell)=0$ \.
and \. $\aF(k,\ell+2) > 0$ \. follows via poset duality as in the proof above.
Finally, the nonvanishing case \. $\aF(k,\ell+2)=\aF(k+2,\ell)>0$ \.
is given by the second inequality in Theorem~\ref{t:thin-part1}.
\end{proof}

\smallskip

\begin{proof}[Proof of Theorem~\ref{t:main-converse}]
Lemma~\ref{l:Fkell-bound2} combined with~\eqref{eq:Fk-bound1}, gives
$$\frac{\aF(k+1,\ell)}{\aF(k,\ell) } \, \leq \, 2k\ell.$$
Similarly, Lemma~\ref{l:Fkl-bound3} for \. $k'=k-1$ \. and \. $\ell'=\ell+1$,
combined with~\eqref{eq:Fk-bound2}, gives:
$$
\frac{\aF(k,\ell+1)}{\aF(k+1,\ell+1)} \, = \, \frac{\aF(k'+1,\ell')}{\aF(k'+2,\ell')} \, \leq \, n \ell' \, = \, n(\ell+1).
$$
Multiplying these inequalities, we obtain the first term in the minimum of
the desired upper bound. Via poset duality,
see the proof of Theorem~\ref{t:main} above, we can exchange the $k$ and $\ell$ terms and obtain the other inequality.
\end{proof}

\bigskip

\section{Examples and counterexamples}\label{s:explicit}

\subsection{Inequalities \eqref{eq:CPC-2} and \eqref{eq:CPC-3}} \label{ss:explicit-CPC23}
Recall that by Theorem~\ref{t:CPC-two-three} at least one of these two inequalities
must hold.  We now show that for some posets \eqref{eq:CPC-3} does not hold.  By the
poset duality, the inequality \eqref{eq:CPC-2} also does not hold.

\smallskip

\begin{prop}
 \label{p:CPC3-false}
The inequality \eqref{eq:CPC-3} fails for an infinite family of posets of width three.
\end{prop}


\begin{proof}
 Fix \. $k \geq 1$ \. and \. $\ell \geq 2$, and let \. $P:=(X,\prec)$ \. be the poset given by
\begin{equation*}
	\begin{split}
		& X \ := \ \{x_1,\ldots, x_{k-1}\} \ \sqcup \ \{y_1,\ldots, y_{\ell-2}\} \  \sqcup \  \{z_1, z_2, z_3 \} \ \  \sqcup \ \ \{u,v,w\}\., \\
		& z_1 \. \prec \. x_1 \. \prec \. x_2 \. \prec \. \cdots \. \prec \.  x_{k-1} \. \prec \. z_2 \. \prec \. y_1 \. \prec \. y_2  \. \prec \. \cdots \. \prec \. y_{\ell-2} \. \prec \. z_3\.,\\
		& x_{k-1} \. \prec u \. \prec \. y_1\.,  \ \ v \. \succ \. z_2\., \ \ w \. \succ \. z_2\..
	\end{split}
\end{equation*}
Note that this is a poset of width three.  Let us now compute all four terms in~\eqref{eq:CPC-3}:

\smallskip

First, observe that \. $L\in \cF(k,\ell+2)$ \. if and only if
\. $L(z_2) < L(u) < L(y_1)$ \. and \. $L(v), L(w)<L(z_3)$.  Thus, there is a bijection
between these linear extensions and the pairs \. $(i,j)$ \. satisfying \.
$1 \leq i \neq j \leq \ell+1$, through the map \. $L \mapsto \big(L(v)-L(z_2), L(w)-L(z_2)\big)$.
Therefore, we have \. $\aF(k,\ell+2)=(\ell+1)\ell$.

Second, observe that \. $L\in \cF(k+1,\ell)$ \. if and only if \. $L(x_{k-1}) < L(u)<L(z_2) $,
and either \. $L(v)< L(z_3) <L(w)$ \. or \. $L(w)< L(z_3) <L(v)$.
Note that there is a bijection between those linear extensions satisfying
\. $L(v)< L(z_3) <L(w)$ \. and the integers in $[1,\ell-1]$, through the map
$L\mapsto L(v)-L(z_2)$.  Therefore, we have \. $\aF(k+1,\ell)=2(\ell-1)$.

Third,  observe that \. $L\in \cF(k,\ell+1)$  \. if and only if \.  $L(z_2) < L(u) < L(y_1)$,
and either \. $L(v)< L(z_3) <L(w)$ \. or \. $L(w)< L(z_3) <L(v)$.
Note that there is a bijection between those linear extensions satisfying
\. $L(v)< L(z_3) <L(w)$ \. and the integers in $[1,\ell]$, through the map
$L\mapsto L(v)-L(z_2)$. Therefore, we have \. $\aF(k,\ell+1)=2\ell$.

Fourth,  observe that \. $L\in \cF(k+1,\ell+1)$  \. if and only if \. $L(x_{k-1}) < L(u)<L(z_2)$ \. and  \. $L(v),L(w)<L(z_3)$.
Note that there is a bijection between these linear extensions  and pairs \. $(i,j)$ \. satisfying \.
$1 \leq i \neq j \leq \ell$. Therefore, we have \. $\aF(k+1,\ell+1)=\ell (\ell-1)$.

\smallskip

Combining these observations, we obtain:
\[ \frac{\aF(k,\ell+1) \. \aF(k+1,\ell+1)}{\aF({k,\ell+2}) \. \aF(k+1,\ell)} \ = \ \frac{\ell}{\ell+1} \ < \ 1\ts.    \]
This contradicts \eqref{eq:CPC-3}, as desired.
\end{proof}

\subsection{Counterexamples to the generalized CPC} \label{ss:explicit-gen-CPC}
We now show that the examples in proof of Proposition~\ref{prop:imply} are also
counterexamples to Conjecture~\ref{conj:GCPC}, thus proving Theorem~\ref{t:GCPC-false}.

\smallskip

\begin{prop} \label{prop:imply}
Inequality \eqref{eq:GCPC} implies \eqref{eq:CPC-3}.
\end{prop}


\begin{proof}
Suppose \eqref{eq:CPC-3} fails for a poset \. $P=(X,\prec)$, elements \. $z_1, z_2, z_3\in X$, and integers \.
$k,\ell \ge 1$.

Let \. $z_1':=z_2$, \. $z_2':=z_1$, and \. $z_3':=z_3$.  To avoid the clash of notation, let
\. $\aF'(k,\ell)$ \. be defined by
\[ \aF'(k,\ell) \ := \ \big|\{ \ts L \in \Ec(P) \ : \ L(z_2')-L(z_1')=k, \, \. L(z_3') -L(z_2') = \ell \. \}\big|.  \]
By definition, we have
\[ \aF'(a,b) \ = \  \aF(-a,a+b). \]
Now let \. $a:=-k-1$ \. and \. $b:=\ell+k+1$.  Note aside that \. $a<0$ \. for all \. $k >0$. It then follows that
\begin{align*}
	\aF'(a,b) \ &= \ \aF(-a,a+b) \ = \ \aF(k+1,\ell), \\
		\aF'(a+1,b+1) \ &= \ \aF(-a-1,a+b+2) \ = \ \aF(k,\ell+2), \\
			\aF'(a,b+1) \ &= \ \aF(-a,a+b+1) \ = \ \aF(k+1,\ell+1), \\
					\aF'(a+1,b) \ &= \ \aF(-a-1,a+b+1) \ = \ \aF(k,\ell+1).
\end{align*}
In the new notation, the inequality \eqref{eq:CPC-3} is equivalent to
\[  \aF'(a,b) \. \aF'(a+1,b+1) \  \leq \ \aF'(a,b+1) \. \aF'(a+1,b), \]
and note that \.$a<0$, $b>0$ \. whenever \. $k,\ell >0$.
This shows that a counterexample for \eqref{eq:CPC-3} is also a counterexample to \eqref{eq:GCPC}.
\end{proof}

\smallskip

\begin{cor} \label{c:CPC23-width-two}
Inequalities \eqref{eq:CPC-2} and \eqref{eq:CPC-3} hold for posets of width two.
\end{cor}

\smallskip

This follows from Proposition~\ref{prop:imply} and Theorem~3.3 in~\cite{CPP1} which proves
\eqref{eq:GCPC} for posets of width two.

\smallskip

\subsection{Stanley ratio}\label{ss:explicit-Stanley}
It follows from Corollary~\ref{cor:Sta-upb}, the following bound on the
\defn{Stanley ratio}:
\begin{equation}\label{eq:Stanley-upper}
\frac{\aN_k^2}{\aN_{k-1}\. \aN_{k+1} } \ \leq \ {(k-1)(n-k)},
\end{equation}
whenever the LHS is well defined.  The following example shows that
both the inequality~\eqref{eq:Stanley-upper} and Corollary~\ref{cor:Sta-upb}
are tight.

\smallskip

In the notation of~$\S$\ref{ss:words-single}, fix \. $1\le k \le n$.
Let \. $P_k:=(X,\prec)$ \. be the width two poset given by
\begin{equation*}
	\begin{split}
		& X \ := \ \{x_1,\ldots, x_{k-2}\} \ \sqcup \ \{y_1,\ldots, y_{n-k-1}\} \  \sqcup \  \{a, v, w \}, \\
		& x_1 \. \prec \. x_2 \. \prec \. \cdots \. \prec \.  x_{k-2} \. \prec \. a \. \prec \. y_1 \. \prec \. y_2  \. \prec \. \cdots \. \prec \. y_{n-k-1},\\
		& v \. \prec \. y_1\ts,  \ \ w \. \succ \. x_1\ts, \ \ v \. \prec \. w.
	\end{split}
\end{equation*}

\smallskip

\begin{prop} \label{prop:Stanley-up}
For posets \. $P_k$ \. defined above  the inequality \eqref{eq:Stanley-upper} is an equality. 
\end{prop}

\begin{proof}
Note that for all linear extensions \. $L \in \Nc_{k-1}$\ts, we have \. $L(a) < L(v)=k < L(w)$,
where \. $k+1 \le L(w) \le n$.
Similarly, for all linear extensions \. $L \in \Nc_{k}$\ts, we have \. $L(v) <L(a) < L(w)$,
where \. $1 \le L(v) \le k-1$ \. and \. $k+1\le L(w) \le n$.  Finally,
for all linear extensions \. $L \in \Nc_{k+1}$\ts, we have \.  $L(v)< L(w)=k < L(a)$, where
\. $1 \le L(v) \le k-1$.  These three observation imply thati
\begin{equation*}\label{eq:Sta-up-calc}
	\aN_{k-1} \ = \  n-k, \qquad \aN_{k} \ = \  (k-1) (n-k), \qquad \aN_{k+1} \ = \  k-1.
\end{equation*}
Thus, for posets~$P_k$ 
the inequality~\eqref{eq:Stanley-upper} is an equality.
\end{proof}


\smallskip

\subsection{Converse cross product ratio}\label{ss:explicit-converse}
The following example shows that Theorem~\ref{t:main-converse}
is essentially tight, up to a multiplicative factor of \ts $2\ell$.
%
Fix \. $k \geq 2$,  \. $\ell \geq 1$, and denote \ts $m:=n-k-\ell-3$.
Let \. $P_{k,\ell}:=(X,\prec)$ \. be the poset given by
\begin{equation*}
	\begin{split}
		& X \ := \ \{a_1,\ldots, a_{k-2}\} \ \sqcup \  \{b_1,\ldots, b_{\ell-1}\} \  \sqcup \ \{c_1,\ldots, c_{m}\} \ \sqcup \  \{z_1, z_2, z_3 \} \   \sqcup  \ \{u,v,w\}, \\
		& z_1 \. \prec \. a_1 \.  \prec \. \cdots \. \prec \.  a_{k-2} \. \prec \. z_2 \. \prec \. b_1  \. \prec \. \cdots \. \prec \. b_{\ell-1} \. \prec \. z_3 \. \prec \. c_1 \. \prec \. \cdots \. \prec \. c_m,\\
		& u \. \prec  \.  z_2\ts,  \ \ a_{k-2} \. \prec \. v \. \prec \. z_3\ts, \ \ w \. \succ \. b_{\ell-1}\ts, \ \ u \. \prec v \. \prec \. w.
	\end{split}
\end{equation*}

\smallskip

\begin{prop}\label{prop:converse}
Fix \. $k \geq 2$,  \. $\ell \geq 1$.  For posets \. $P_{k,\ell}$ \. defined above, we have: 
\begin{equation}\label{eq:sta-converse}
 \frac{\aF(k,\ell+1) \. \aF(k+1,\ell)}{\aF(k,\ell) \. \aF(k+1,\ell+1)}
\ = \ k\ell n \big(1+o(1)\big)  \ \  \text{as \ $n\to \infty$.}
\end{equation} \end{prop}


\begin{proof}
Note that for every linear extension \. $L\in \Ec(P_{k,\ell})$,  we have: 
\begin{align}\label{eq:dam-1}
	L(z_2)-L(z_1) \  &\geq \ |B(z_1,z_2) \ts - \ts z_1| \ = \   |\{a_1,\ldots, a_{k-2}, z_2\}| \ = \ k-1,\\
	\label{eq:dam-2}	L(z_3)-L(z_2) \  &\geq  \ |B(z_2,z_3) \ts - \ts z_2| \ =  \ |\{b_1,\ldots, b_{\ell-1}, z_3\}| \ = \ \ell.
\end{align}
Note also that
\begin{align}
	\label{eq:dam-3}
 \text{either} \quad L(u)=1  \quad &\text{or} \quad L(z_1) < L(u)< L(z_2),\\
\label{eq:dam-4}		\text{either} \quad L(v)=L(z_2)-1  \quad & \text{or} \quad L(z_2) < L(v)< L(z_3),\\
\label{eq:dam-5}		\text{either} \quad L(w)= L(z_3)-1  \quad & \text{or} \quad L(w) > L(z_3).
\end{align}

We now compute the cross--product ratio of \. $P_{k,\ell}$ \. consider the following four cases.

\smallskip

\nin
\textbf{Case 1.} 
 Let  $L \in \cF(k,\ell)$.
Since $L(z_3)-L(z_2)=\ell$, it then follows from \eqref{eq:dam-2} that both  \. $L(v)$ \. and \. $L(w)$ are not contained in the interval \. $[L(z_2),L(z_3)]$\..
It then follows from \eqref{eq:dam-4} and \eqref{eq:dam-5} that
\.  $L(v)= L(z_2)-1$ \.  and \.  $L(w)>L(z_3)$\., respectively.
Now, since \. $L(z_2)-L(z_1)=k$ \. and \. $L(v)\in \big[L(z_1),L(z_2)\big]$\., it then follows from \eqref{eq:dam-1} that \. $L(u)$ \. is not contained in the interval \. $[L(z_1),L(z_2)]$\..
It then follows from \eqref{eq:dam-3} that
\. $L(u)=1$\, which in turn implies that $L(z_1)=2$.
We conclude that \. $L \in \cF(k,\ell)$ \. satisfy:
\begin{alignat*}{3}
&L(z_1)=2, \quad && L(z_2)=k+2, \quad &&L(z_3)=k+\ell+2, \\
& L(u)=1, \quad && L(v)=k+1, \quad  && L(w) \in [k+\ell+3,n].
\end{alignat*}
This implies that \. $\aF(k,\ell)=n-k-\ell-2$\., as desired.

\smallskip

\nin
\textbf{Case 2.} 
 Let  $L \in \cF(k+1,\ell)$.
Since $L(z_3)-L(z_2)=\ell$, it then follows from \eqref{eq:dam-2} that both  \. $L(v)$ \. and \. $L(w)$ are not contained in the interval \. $[L(z_2),L(z_3)]$\..
It then follows from \eqref{eq:dam-4} and \eqref{eq:dam-5} that
\.  $L(v)= L(z_2)-1$ \.  and \.  $L(w)>L(z_3)$, respectively.
Now, since \. $L(z_2)-L(z_1)=k+1$ \. and \. $L(v)\in \big[L(z_1),L(z_2)\big]$\., it then follows from \eqref{eq:dam-1} that \. $L(u)$ \. is  contained in the interval \. $[L(z_1),L(z_2)]$\..
It then follows that $L(z_1)=1$.
We conclude that \. $L \in \cF(k+1,\ell)$ \.  satisfy:
\begin{alignat*}{3}
	&L(z_1)=1, \quad && L(z_2)=k+2, \quad &&L(z_3)=k+\ell+2, \\
	& L(u)\in [2,k], \quad && L(v)=k+1, \quad  && L(w) \in [k+\ell+3,n].
\end{alignat*}
This implies that \. $\aF(k+1,\ell)=(k-1)(n-k-\ell-2)$.

\smallskip

\nin
\textbf{Case 3.} We have \.  $\aF(k,\ell+1)=1+ (k-1) \, \ell (n-k-\ell-2)$ \. by the following argument.
 Let  $L \in \cF(k,\ell+1)$.
By \eqref{eq:dam-3} either  \. $L(u)=1$  \. or \. $L(u) \in [L(z_1), L(z_2)]$\..

\smallskip

\nin
\textbf{Case 3.1} Assume that \. $L(u)=1$\.. This implies that $L(z_1)=2$.
Since $L(u) \notin [L(z_1),L(z_2)]$\.,
it then follows from \. $L(z_2)-L(z_1)=k$ \. and \eqref{eq:dam-1} that \. $L(v)$ \. is contained in the interval $[L(z_1), L(z_2)]$.
It then follows from  \eqref{eq:dam-4}
that $L(v)=L(z_2)-1$.
Since \. $L(z_3)-L(z_2)=\ell+1$ \. and \. $L(v)\notin [L(z_2),L(z_3)]$, it then follows from \eqref{eq:dam-2} that
\. $L(w)$ \. is contained in the interval \.  $[L(z_2),L(z_3)]$\..
By \eqref{eq:dam-5}, this implies that \. $L(w)=L(z_3)-1$.
We conclude:
\begin{alignat*}{3}
	&L(z_1)=2, \quad && L(z_2)=k+2, \quad &&L(z_3)=k+\ell+3, \\
	& L(u)=1, \quad && L(v)=k+1, \quad  && L(w) =k+\ell+2.
\end{alignat*}
Thus, there is exactly one such linear extension.

\smallskip

\nin
\textbf{Case 3.2} Assume that \. $L(u) \in [L(z_1), L(z_2)]$\..
This implies that $L(z_1)=1$.
Since \. $L(z_2)-L(z_1)=k$ \. and \. $L(u) \in [L(z_1), L(z_2)]$, it then follows from \eqref{eq:dam-1} that
\. $L(v)$ \. is not contained in the interval \. $[L(z_1), L(z_2)]$.
By \eqref{eq:dam-4}, this implies that \. $L(v)$ \. is contained in the interval
 \. $[L(z_2), L(z_3)]$.
Since  \. $L(z_3)-L(z_2)=\ell+1$\., it then follows from \eqref{eq:dam-2} that $L(w)$
is not contained in the interval
 \. $[L(z_2), L(z_3)]$.
 By \eqref{eq:dam-5}, this implies that $L(w)>L(z_3)$.
 We conclude: 
 \begin{alignat*}{3}
 	&L(z_1)=1, \quad && L(z_2)=k+1, \quad &&L(z_3)=k+\ell+2, \\
 	& L(u)\in[2,k], \quad && L(v)\in [k+2,k+\ell+1], \quad  && L(w) \in [k+\ell+3,n].
 \end{alignat*}
Thus, there are exactly \. $(k-1)\ell(n-k-\ell-2)$ \. such linear extensions. 

\smallskip

\nin
\textbf{Case 4.} 
 Let \. $L \in \cF(k+1,\ell+1)$.
 Since \. $L(z_2)-L(z_1)=k+1$, it follows from \eqref{eq:dam-1} 
 that both \. $L(u)$ and $L(v)$ \. are contained in the interval \. $[L(z_1),L(z_2)]$.
 This implies that \. $L(z_1)=1$.
 Since \. $L(v) \in [L(z_1),L(z_2)]$, it then follows from~\eqref{eq:dam-4} that \. $L(v)=L(z_2)-1$.
 Now, since  \. $L(z_3)-L(z_2)=\ell+1$ \. and \. $L(v) \notin [L(z_2), L(z_3)]$,
 it then follows  from~\eqref{eq:dam-2} that \ts $L(w)$ \ts is  contained in the interval
\. $[L(z_2), L(z_3)]$.
By \eqref{eq:dam-5} this implies that \. $L(w)=L(z_3)-1$.
We conclude that \. $L \in \cF(k+1,\ell+1)$ \. satisfy: 
\begin{alignat*}{3}
	&L(z_1)=1, \quad && L(z_2)=k+2, \quad &&L(z_3)=k+\ell+3, \\
	& L(u)\in [2,k], \quad && L(v)=k+1, \quad  && L(w)=k+\ell+2.
\end{alignat*}
This implies that \. $\aF(k+1,\ell+1)=k-1$.

\smallskip

In summary, for the poset~$P_{k,\ell}$\ts, we have:
\[  \frac{\aF(k,\ell+1) \. \aF(k+1,\ell)}{\aF(k,\ell) \. \aF(k+1,\ell+1)}
\ = \   1+ (k-1) \ell (n-k-\ell-2) \ = \ k\ell n \big(1+o(1)\big)  \ \  \text{as \ $n\to \infty$,}
\]
as desired.
\end{proof}

\bigskip

\section{Final remarks} \label{s:finrem}

\subsection{}\label{ss:finrem-CPC}
The \defng{cross-product conjecture} (Conjecture~\ref{conj:CPC})
has been a major open problem
in the area for the past three decades, albeit with relatively little progress
to show for it (see \cite{CPP1} for the background).  The following quote
about a closely related problem seems applicable:

\smallskip

\begin{center}\begin{minipage}{12cm}%
{\emph{``As sometimes happens, we cannot point to written evidence that
the problem has received much attention; we
can only say that a number of conversations over the last 10 years suggest that the
absence of progress on the problem was not due to absence of effort.''}~\cite[p.~87]{KY}}
\end{minipage}\end{center}

\subsection{}\label{ss:finrem-hist}
Theorem~\ref{thm:Favard} is well-known in the area and can be traced back
to the works of Jean Favard in the early 1930s.\footnote{Ramon van Handel,
personal communication (May 3, 2021).}  Of course, this is not the only
\defng{Favard's inequality} \ts known in the literature.  In fact,
Theorem~\ref{thm:Sch} which goes back to Matsumura (1932) and Fenchel (1936),
seem to also have been inspired by Favard's work.\footnote{Ramon van Handel,
personal communication (June 12, 2023).} In a closely related context of
\defng{Lorentzian polynomials}, Favard's inequality appears in
\cite[Prop.~2.17]{BH}.  For more on Theorem~\ref{thm:Sch},
see \cite[$\S$51]{BF87} and references therein.

\subsection{} \label{ss:finrem-needle}
As we mentioned in the introduction, the \ts $\Ups\ge 1/2$ \ts
lower bound derived from Favard's inequality (Theorem~\ref{thm:Favard})
easily implies the $1/2$ lower bound on the cross product.  Given the
straightforward nature of this implication, one can think of this paper
as the first attempt to finding the best \. $\ve\ge 0$, such that
$$\frac{\aF(k+1,\ell) \. \aF(k,\ell+1)}{\aF(k,\ell)\. \aF(k+1,\ell+1)} \, - \, \frac{1}{2} \ \geq \, \ve.
$$
In this notation, the CPC states that \. $\ve=\frac12$.  Our Main Theorem~\ref{t:main}
and especially the ``$t$-thin or $t$-flat'' Theorem~\ref{t:thin} are the first effective
bounds for~$\ve>0$.  More precisely, here we prove \. $\ve = \Omega\big(\frac{1}{n}\big)$ \.
for all posets, and a constant lower bound on \ts $\ve$ \ts  for posets with bounded parameter~$t$.
Improving these bounds seems an interesting challenging
problem even if the CPC ultimately fails.

We should mention that from the Analysis point of view, the first explicit bound
on \. $\ve>0$ \. is typically an important step, no matter how small.  We refer
to \cite{KLT} and \cite{CGMS} for especially remarkable breakthroughs of this type.

\subsection{} \label{ss:finrem-other}
The constant \ts $1/2$ \ts in Favard's inequality has the same nature as the constant~$2$
in \cite[Cor.~1.5]{RSW} which also follows from Favard's inequality written in terms of 
the Lorentzian polynomials technology.  The
relationships to the constant~2 in \cite[Thm~1.1]{CP3}  and \cite[Thm~5]{HSW} are more
distant, but fundamentally of the same nature.  While in the former case it is tight,
in the latter is likely much smaller, see \cite[$\S$2.3]{Huh}.

\subsection{}\label{ss:finrem-discremancy}
The reader might find surprising the discrepancy between the vanishing and the
nonvanishing cases in the Main Theorem~\ref{t:main}.
Note that the vanishing case actually implies a \emph{worse bound} \eqref{eq:main-thm-2}
compared to the bound \eqref{eq:main-thm-1} in the nonvanishing case, instead of
making things simpler. This is an artifact of the mixed volume inequalities and combinatorial ratios.
Proposition~\ref{p:Fk-bound1} gives a better bound than Proposition~\ref{p:Fk-bound2}
simply because the ratio of \ts $\aF(\cdot,\cdot)$'s \ts in the former is under
a square root which decreases the order. However, these combinatorial bounds can only
be applied when the corresponding terms are nonzero.

Clearly, there is no way to justify this discrepancy, as otherwise we would know
how to disprove the CPC.  Still, one can ask if there is another approach to the
vanishing case which would improve the bound?  We caution the reader that sometimes
nonvanishing does indeed make a difference (see e.g.\ Example~\ref{ex:vanish-unequal}).

\subsection{}\label{ss:finrem-DD}
Theorem~\ref{thm:stanley_vanish} gives the vanishing conditions for the
\defng{generalized Stanley inequalities}.  It was first stated without a proof in
\cite[Thm~8.2]{DD85}, and it seems the authors were aware of a combinatorial
proof by analogy with their proof of the corresponding results for the order polynomial.
The theorem was rediscovered in \cite[Thm~1.12]{CPP}, where it was proved via combinatorics
of words.  Independently, it was also proved in \cite[Thm~5.3]{MS22} by a
geometric argument.

\subsection{}\label{ss:finrem-pin}
There is a large literature on the negative dependence in a combinatorial
context, see e.g.\ \cite{BBL,Huh,Pem}, and in the context of linear
extensions~\cite{KY,She82}.
When it comes to correlation inequalities for linear extensions of posets,
this paper can be viewed as the third in a series after \cite{CP3} and~\cite{CP4}
by the first two authors.  These papers differ by the tools involved. In~\cite{CP3},
we use the \defng{combinatorial atlas} \ts technology (see~\cite{CP1,CP2}),
while in~\cite{CP4} we use the \defng{FKG-type inequalities}.

The idea of this paper was to use geometric inequalities for mixed volumes,
to obtain new cross product type inequalities.  As we mentioned in the
introduction, it transfers the difficulty to the \defng{combinatorics
of words}.  This is the approach introduced in \cite{Hai,MR} (see also~\cite{Sta-promo}),
and advanced in \cite{CPP1,CPP,CPP2} in a closely related context.

\subsection{} \label{ss:finrem-CS}
Despite the apparent symmetry between the $t$-thin and $t$-flat notions,
there is a fundamental difference between them.  For posets $P=(X, \prec)$
which are $t$-thin with respect to a set~$A$ of bounded size, the number $e(P)$ of linear extension
can be computed in polynomial time, since \ts $P':=(X\sm A,\prec)$ \ts
has width at most~$t$.  On the other hand, for posets which
have bounded height, computing \ts $e(P)$ \ts is  $\SP$-complete \cite{BW,DP},
and the same holds for posets which are $t$-flat with respect to a set 
of bounded size.

\vskip.6cm

\subsection*{Acknowledgements}
We are grateful to Jeff Kahn, Yair Shenfeld and
Ramon van Handel for many helpful discussions and remarks on the subject,
and to Julius~Ross for telling us about~\cite{RSW}.  We are thankful to
Per Alexandersson for providing {\tt Mathematica} packages to test the conjectures.
The first author was partially supported by the Simons Foundation.
The second and third authors were partially supported by the~NSF.



\end{document}